\documentclass[12pt]{amsart}
 \usepackage[T1]{fontenc}
\usepackage{amsmath,amssymb,amsfonts,amsthm,amsopn}
\usepackage{graphics}
\usepackage{xcolor}
\usepackage{pb-diagram}
\usepackage[title]{appendix}
\newcommand{\tf}{time-frequency}

\newcommand{\tfs}{time-frequency shift}

\newtheorem{theorem}{Theorem}[section]
\newtheorem{lemma}[theorem]{Lemma}
\newtheorem{corollary}[theorem]{Corollary}
\newtheorem{proposition}[theorem]{Proposition}
\newtheorem{definition}[theorem]{Definition}

\newtheorem{example}[theorem]{Example}
\newtheorem{remark}[theorem]{Remark}

\newcommand{\beqa}{\begin{eqnarray*}}
\newcommand{\eeqa}{\end{eqnarray*}}

\DeclareMathOperator*{\diag}{diag}
\DeclareMathOperator*{\Sym}{Sym}
\DeclareMathOperator*{\eig}{eig}
\DeclareMathOperator*{\Sp}{{Sp}}

\newcommand{\field}[1]{\mathbb{#1}}
\newcommand{\bR}{\field{R}}        %  real numbers
\newcommand{\bC}{\field{C}}        %  complex numbers
\def\la{\lambda}

\def\cF{\mathcal{F}}              % Calligraphic Letters
\def\cS{\mathcal{S}}
\def\cD{\mathcal{D}}

\def\cB{\mathcal{B}}

\def\cM{\mathcal{M}}

\def\cA{\mathcal{A}}

\def\cI{\mathcal{I}}

\def\rd{\bR^d}

\def\rdd{{\bR^{2d}}}

\def\lrd{L^2(\rd)}

\def\intrd{\int_{\rd}}
\def\intrdd{\int_{\rdd}}

\def\R{\right)}

\def\<{\left<}
\def\>{\right>}

\def\mv1{M_v^1}

\def\phas{(x,\xi )}
\def\mn{(m,n)}
\def\mn'{(m',n')}

\newcommand{\norm}[1]{\lVert#1\rVert}

\def\o{\xi}

\def\R{\mathbb{R}}
\def\Ren{\mathbb{R}^d}

\def\sch{\mathcal{S}}

\def\f{\varphi}

\def\Sn2{S_{2}(L^{2}(\Ren))}
\def\S1{S_{1}(L^{2}(\Ren))}
\def\sig00{\sigma_{0,0}}
\def\la{\langle}
\def\ra{\rangle}

\begin{document}
\begin{abstract} 
To overcome the impossibility of representing the energy of a signal simultaneously in time and frequency, many time-frequency representations have been introduced in the literature. Some of these are recalled in the Introduction.
%In particular, a wide class of repsentations displaying many good properties is the Cohen's class, first proposed by L. Cohen in \cite{Cohen1} and then widely studied by many authors. A proper sub-class is constituted by the so-called generalized spectrograms, introduced by Boggiatto, De Donno, and Oliaro in \cite{BdDO3,BDOJMAA2007,BdDO1}, with fruitful applications in signal analysis. In fact, the previous authors showed in \cite{BdDO2} that generalized spectrograms  reduce the so-called \emph{ghost frequencies}, displayed by the popular Wigner distribution.  The perturbation of the latter has led to the introduction of several bilinear mappings  \cite{BCGT2020,CT2020,toftbil17}.

In this work we propose a unified approach of the previous theory by means of metaplectic Wigner distributions  $W_\cA$, with $\cA$ symplectic matrix in $Sp(2d,\bR)$, which were introduced by Cordero, Rodino (2022)  and then widely studied in  subsequent papers.
Namely, the short-time Fourier transform and the most popular members of the Cohen's class can be represented via metaplectic Wigner distributions. In particular, 
 we introduce $\cA$-metaplectic spectrograms which contain the classical ones and their variations 
  arising from the $\tau$-Wigner distributions of Boggiatto, De Donno, and Oliaro (2010).
 
 We provide a complete characterization of those $\cA$-Wigner distributions which give rise to generalized spectrograms. This characterization is related to the block decomposition of the symplectic matrix $\cA$.  Moreover, a characterization of the $L^p$-boundedness of both $\cA$-Wigner distributions and related metaplectic pseudodifferential operators is provided.
\end{abstract}

\title[A unified approach to time-frequency representations]{A unified approach to time-frequency representations and Generalized Spectrograms}

\author{Elena Cordero}
\address{Universit\`a di Torino, Dipartimento di Matematica, via Carlo Alberto 10, 10123 Torino, Italy}
\email{elena.cordero@unito.it}
\author{Gianluca Giacchi}
\address{Università di Bologna, Dipartimento di Matematica,  Piazza di Porta San Donato 5, 40126 Bologna, Italy; Euler Institute, Università della Svizzera italiana, East Campus, Sector D, Via la Santa 1, 6962 Lugano, Switzerland.}
\email{gianluca.giacchi2@unibo.it}
\author{Luigi Rodino}
\address{Universit\`a di Torino, Dipartimento di Matematica, via Carlo Alberto 10, 10123 Torino, Italy}
\email{luigi.rodino@unito.it}
\thanks{}

%\subjclass{35S05,35S30,
%47G30, 42C15}
%\date{}
\maketitle

\begin{center}
	\textit{Dedicated to Karlheinz Gr\"ochenig on the occasion of his 65th birthday}
\end{center}

\section{Introduction}

The classical method for analyzing a signal $f$ with time-varying frequency content is to multiply it by the translation $T_x \f$ of a window function $\f$ essentially localized in a neighbourhood of the origin and then taking the Fourier transform of the modified signal.
In other words, we consider the  \emph{Short-time Fourier Transform} (STFT) of $f$, formally defined as
\begin{equation}\label{stftdef}
	V_\f f\phas=\intrd f(t)\overline{\f(t-x)}\,e^{-2\pi i t \cdot \o}dt,\quad x,\o\in\rd,
\end{equation}

In the STFT analysis, one intends to achieve the best time and frequency resolutions. 
However, the resolution in the time domain is limited by the width of the window function $\f$
 and similarly, the resolution in the frequency domain is limited by the width of its Fourier transform, $\hat \f$. By Heisenberg's uncertainty principle,  $\f$ and $\hat \f$ cannot be both well localized, thus the width of the window function  $\f$ plays a fundamental role:  a window $\f$ well localized in time  gives good resolution in time and  necessarily implies poor frequency design  whereas  a long time window $\f$ yields poor time resolution but good frequency design.

 The squared magnitude of the STFT is called \emph{the spectrogram}:
 \begin{equation}\label{Spectrogtam}
 		{\Sp}_\f(f)=|V_\f f|^2.
 \end{equation}
The spectrogram is a visual representation of the spectrum of frequencies of a signal as it varies with time \cite{AGR2016,Janssen91}.  It displays poor time-frequency localization properties due to the presence of the analysis window function $\f$.  In order to avoid the time-frequency resolution trade-off problem of the spectrogram, many other time-frequency representations have been introduced in the literature. Among the many contributions on the topic, we recall \cite{Cohen1,Cohen2,Elena-book,FollandSit89,GalleaniCohen,Gos11,book}.
One of the most popular one is the Wigner distribution, introduced by Wigner \cite{Wigner32} in the context of Quantum Mechanics and first applied to signal analysis by  Ville \cite{Ville48}. After that, there have been so many works on the topic that it is impossible to cite them all. 
\begin{definition} Consider $f,g\in\lrd$. 
	The cross-Wigner distribution $W(f,g)$ is
	\begin{equation}\label{CWD}
		W(f,g)(x,\xi)=\intrd f(x+\frac t2)\overline{g(x-\frac t2)}e^{-2\pi i t\cdot \xi}\,dt.
	\end{equation} If $f=g$ we write $Wf=W(f,f)$, the  Wigner distribution of $f$.
\end{definition}
The central role of this representation is provided by the large number of desirable properties it enjoys, see the textbooks \cite{KB2020,Cohen2,deGossonWigner,book,hlawbook,WongWeylTransform1998}.

On the other hand, due to the uncertainty principle again, the Wigner distributions presents drawbacks as well. For example, it lacks of positivity and its interpretation as a phase-space distribution of the signal energy density is troublesome, see the results on Hudson's Theorem \cite{hudson74,janssenhuds84}.

This is the main reason for the introduction of the Cohen's class \cite{Cohen1}
\[
Q(f,g)=W(f,g)*\sigma, \quad f,g\in\cS(\rd),\]
which is obtained by convolving the cross-Wigner distribution with a suitable kernel function $\sigma$, called the Cohen kernel, with the aim of preserving the good properties of the cross-Wigner distribution while dumping the undesirable ones. The  Cohen's class contains many popular time-frequency representations and can be viewed as a unifying framework for their study.
Among these representations we mention the $\tau$-Wigner distributions, $\tau\in\bR$,
defined as
\begin{equation}\label{tauWigner}
	W_\tau(f,g)(x,\xi)=\int_{\rd} f(x+\tau t)\overline{g(x-(1-\tau)t)}e^{-2\pi i\xi \cdot t}dt, 
\end{equation}
for $f,g\in L^2(\rd)$. The case $\tau=1/2$ is the Wigner one.  In high-dimensional complex features information processing $\tau$-Wigner distributions   play a crucial role \cite{ZJQ21}, but they have been investigated in many other frameworks, see \cite{bastianoni2023tauquantization,boghuds, BdDO2,cdet18,cnt18,janssenposspread97}.   Their  Cohen kernel is given by (cf. \cite[Proposition 5.6]{BdDO2}):
\begin{equation}\label{kerneltau}
	\sigma_{\tau}\left(x,\o\right)=\frac{2^{d}}{\left|2\tau-1\right|^{d}}e^{2\pi i\frac{2}{2\tau-1}x\cdot \o}, \quad (x,\o)\in\rdd, \quad \tau \in \bR\setminus\{ 1/2\}.
\end{equation}

The interpretation of $\tau$ as a perturbation parameter of the Wigner transform has led to natural generalizations of the Wigner distributions using an arbitrary linear mapping of the pair $(x,y)\in\rdd$, cf \cite{bcoquadratic,BCGT2020,CT2020,toftbil17}. Namely, 
\begin{equation}\label{BAe}
	\mathcal{B}_{\mathbb{A}}\left(f,g\right)\left(x,\o\right)=\int_{\mathbb{R}^{d}}f\left(\mathbb{A}_{11}x+\mathbb{A}_{12}y\right)\overline{g\left(\mathbb{A}_{21}x+\mathbb{A}_{22}y\right)}e^{-2\pi i\o\cdot y}dy,
\end{equation}

where $\mathbb{A}=\left(\begin{array}{cc}
	\mathbb{A}_{11} & \mathbb{A}_{12}\\
	\mathbb{A}_{21} & \mathbb{A}_{22}
\end{array}\right)$ is a $2d\times 2d$ invertible matrix.

It was shown in \cite[theorem 1.1]{CT2020} that $\mathcal{B}_{\mathbb{A}}$ belongs to the Cohen's class if and only if  $\mathbb{A}$ displays  the following block decomposition: 
\begin{equation}\label{AM}
	\mathbb{A}=\mathbb{A}_{M}=\left(\begin{array}{cc}
		I_{d\times d} & M+(1/2)I_{d\times d}\\
		I_{d\times d} & M-(1/2)I_{d\times d}
	\end{array}\right),
\end{equation}
where $I_{d\times d}$ is the $d\times d$ identity matrix and  $M\in\mathbb{R}^{d\times d}$. In this case, we have 
\begin{equation}\label{Bamtheta}
	\mathcal{B}_{\mathbb{A}_{M}}\left(f,g\right)=W\left(f,g\right)*\theta_{M},
\end{equation}
where the Cohen's kernel $\theta_{M}\in \cS'(\rdd)$ is given by
\begin{equation}\label{thetaM}
	\theta_{M}=\cF_\sigma \chi_M,\quad \mbox{with}\quad \chi_M(x,\xi)=e^{2\pi i \xi\cdot Mx},
\end{equation}
where  $\cF_\sigma$ is the symplectic Fourier transform (cf. \eqref{SyFT} below).

It is well known that if $M$ is invertible, then the symplectic Fourier transform $\cF_\sigma \chi_M$   can be computed explicitly so that
\begin{equation}\label{thM}
	\theta_{M}\left(x,\o\right)=\frac{1}{\left|\det M\right|}e^{2\pi ix\cdot M^{-1}\o}.
\end{equation}
In particular, choosing $M=\left( \tau-1/2\right)I_{d\times d}$, with $\tau\in \bR\setminus \{1/2\}$, we recapture the $\tau$-kernels in \eqref{kerneltau}.

%In this work we compute explicitly $\theta_{M}$ even though $M$ is not invertible, cf. Corollary \ref{corFPhiC} below.

In \cite{CR2021} two of the authors introduced a new  perspective in this field: time-frequency representations can be viewed as images of metaplectic operators, which become the real protagonists in signal analysis, as they have the potential of unify a large part of it, as suggested in \cite{CGshiftinvertible,CGshiftinvertible,CG-Excursus,Fuhr,Giacchi,Zhang21bis,Zhang21}. For a fixed metaplectic operator $\hat\cA\in Mp(2d,\bR)$ and related symplectic matrix $\cA\in Sp(2d,\bR)$, the metaplectic Wigner distribution $W_\cA$ is defined by
\begin{equation}\label{defMetWigintro}
	W_\cA(f,g)=\hat\cA(f\otimes \bar g),\quad f,g\in L^2(\rd).
\end{equation}
We refer to Subsection \ref{subsec:26} for the definition of metaplectic operators and related symplectic matrices, and to Subsection \ref{2.4} for the main properties of $W_\cA$. 

Let us highlight that the STFT and the $\tau$-Wigner representations are examples of $\cA$-Wigner distributions.  Namely, $V_gf=\widehat{A_{ST}}(f\otimes\bar g)$ and $W_\tau(f,g)=\widehat{A_\tau}(f\otimes\bar g)$, where% if $I_{d\times d}$ denotes the identity of $\bR^{d\times d}$ and $0_{d\times d}$ denotes the matrix of $\bR^{d\times d}$ with all zero entries,
\begin{equation}\label{AST}
	A_{ST}=\begin{pmatrix}
		I_{d\times d} & -I_{d\times d} & 0_{d\times d} & 0_{d\times d}\\
		0_{d\times d} & 0_{d\times d} & I_{d\times d} & I_{d\times d}\\
		0_{d\times d} & 0_{d\times d} & 0_{d\times d} & -I_{d\times d}\\
		-I_{d\times d} & 0_{d\times d} & 0_{d\times d} &0_{d\times d}
	\end{pmatrix}
\end{equation}
and
\begin{equation}\label{Atau}
	A_\tau=\begin{pmatrix}
		(1-\tau)I_{d\times d} & \tau I_{d\times d} & 0_{d\times d} & 0_{d\times d}\\
		0_{d\times d} & 0_{d\times d} & \tau I_{d\times d} & -(1-\tau)I_{d\times d}\\
		0_{d\times d} & 0_{d\times d} & I_{d\times d} & I_{d\times d}\\
		-I_{d\times d} & I_{d\times d} & 0_{d\times d} & 0_{d\times d}
	\end{pmatrix}.
\end{equation}
In particular, for $\tau=1/2$, $W(f,g)=\widehat A_{1/2}(f\otimes\bar g)$. 

Also, the bilinear mappings $\mathcal{B}_{\mathbb{A}}$, up to the constant  $|\det \mathbb{A}|$, defined in \eqref{BAe}, are instances of metaplectic Wigner distributions. In fact, any $\mathcal{B}_{\mathbb{A}}$ can be splitted into the product of two  metaplectic operators: the partial Fourier transform with respect to the second variable $\cF_2$ and the change of variables
$$ \mathfrak{T}_\mathbb{A}f:=|\det(\mathbb{A})|^{1/2}\,f(\mathbb{A}\cdot),$$
 see Example \ref{es22} below. In detail, 
$$\mathcal{B}_{\mathbb{A}}(f,g)=|\det \mathbb{A}|^{-1/2}\cF_2\mathfrak{T}_\mathbb{A}(f\otimes\bar g), \qquad f,g\in L^2(\rd).$$
In particular, since $\det \mathbb{A}_M=1$, the Cohen distribution $B_{\mathbb{A}_M}$ in \eqref{Bamtheta} is a metaplectic Wigner distribution.

In conclusion,  the linear perturbations of the Wigner transform introduced by Bayer et al. \cite{BCGT2020}, and later developed in \cite{CT2020}, are special instances of metaplectic Wigner distributions.

A relevant sub-class of the metaplectic Wigner distributions, containing the $\tau$-Wigner representations and their generalizations, is given by those satisfying the Cohen's condition.
By \cite[Theorem 2.11]{CR2022}, a metaplectic Wigner distribution $W_\cA$ belongs to the Cohen's class if and only if
\begin{equation}\label{covMWDE}
	W_\cA(f,g)=W(f,g)\ast \cF^{-1}(e^{-i\pi {\mathcal M}_\cA\zeta\cdot\zeta}).
\end{equation}
To be definite, if the symplectic matrix $\cA$ has block decomposition 
\begin{equation}\label{blockDecA}
	\cA=\begin{pmatrix}
		A_{11} & A_{12} & A_{13} & A_{14}\\
		A_{21} & A_{22} & A_{23} & A_{24}\\
		A_{31} & A_{32} & A_{33} & A_{34}\\
		A_{41} & A_{42} & A_{43} & A_{44}
	\end{pmatrix},
\end{equation}
then $W_\cA$ is in the Cohen's class, that is
\[
	W_\cA(f,g)=\kappa_\cA\ast W(f,g), \qquad f,g\in L^2(\rd),
\] 
for some {\em Cohen kernel} $\kappa_\cA\in\cS'(\rdd)$, if and only $\cA$ is of the form
\begin{equation}\label{A-covariant}
	\cA=\left(\begin{array}{cccc}
		A_{11} & I_{d\times d}-A_{11}&A_{13}&A_{13}\\
		A_{21}&-A_{21}&I_{d\times d}- A^T_{11} &-A^T_{11}\\
		0_{d\times d}&0_{d\times d}&I_{d\times d} &I_{d\times d}\\
		-I_{d\times d}&I_{d\times d}& 0_{d\times d} &0_{d\times d}\\
	\end{array}\right).
\end{equation}
with $A_{13}^T=A_{13}$ and $A_{21}^T=A_{21}$. 

If $W_\cA$ is a covariant metaplectic Wigner distribution as in \eqref{covMWDE}, then ${\mathcal M}_\cA$ is the symmetric matrix 
\begin{equation}\label{defBA}
	{\mathcal M}_\cA=\begin{pmatrix}
		A_{13} & \frac{1}{2}I_{d\times d}-A_{11}\\
		\frac{1}{2}I_{d\times d}-A_{11}^T & -A_{21}
	\end{pmatrix}.
\end{equation}

The Cohen distribution $B_{\mathbb{A}_M}$ defined in \eqref{Bamtheta} falls in the previous class. In fact, we shall show (see Proposition \ref{prop-conn} below) that
\begin{equation}\label{connection}
		\mathcal{B}_{\mathbb{A}_{M}}\left(f,g\right)=W\left(f,g\right)*\theta_{M}=W\left(f,g\right)*\cF^{-1}(e^{-i\pi {\mathcal M}_\cA\zeta\cdot\zeta})
\end{equation}
where 
\begin{equation}\label{MA}
{\mathcal M}_\cA=\begin{pmatrix}
	0_{d\times d} & M\\
	M^T & 	0_{d\times d}
\end{pmatrix},
\end{equation}
and $M$ is the $d\times d$ matrix appearing in \eqref{AM}.

In the next sections we present new results for two other relevant sub-classes of the metaplectic distributions $W_A$, namely shift-invertible distributions and generalized spectrograms. 

The first part of this paper is dedicated to the boundedness properties of new quantizations rules for pseudodifferential operators, defined by means of shift-invertible distributions, introduced in \cite{CR2022}. These distributions were initially defined through a translation-invariance like condition, that we rephrased in the next Theorem \ref{CG-thm1}, finding an equivalent definition that we shall use in the present paper. Relevant examples are the STFT and the $\tau$-Wigner distributions for $\tau\not=0,1$. They enjoy good continuity properties; results on modulation spaces introduced by Feichtinger in \cite{feichtinger-modulation} were proved in \cite{CR2022}. Here, following ideas from \cite{BdDO2}, we shall study boundedness on Lebesgue spaces, see the next Proposition \ref{CGR-cor1}, with related applications to pseudodifferential operators  defined by
$$\la Op_A(a)f,g\ra= \la a,W_A(g,f)\ra,$$
see the next Theorem \ref{CGR-thm1}. 

We now turn our attention to a sub-class of the Cohen's one called \emph{generalized spectrograms}, introduced and studied by Boggiatto et al.    in \cite{BdDO3,BDOJMAA2007,BdDO2,BdDO4}, which extends the classical  one in \eqref{Spectrogtam} and will be the main focus of this manuscript.
\begin{definition}\label{def1.2}
	 The generalized spectrogram with windows $\phi$ and $\psi$ is the time-frequency representation
\begin{equation}\label{spect}
	{\Sp}_{\phi,\psi}(f,g)=V_\phi f\overline{V_\psi g},
\end{equation}
for $\phi,\psi,f,g$ in suitable function or distributional spaces.
\end{definition}
Let $\cI F(x)=F(-x)$ be the flip operator, then it was shown in \cite{BdDO3} that
\begin{equation}\label{gspW}
	V_\phi f \overline{V_\psi g}= W(f,g)\ast W(\cI\psi,\cI\phi).
\end{equation}
Therefore, generalized spectrograms belong to the Cohen's class. In fact, they are a \emph{proper} sub-class: in \cite{BdDO3} it 
was proved that the Wigner representation is not a generalized spectrogram and, more generally in \cite{BdDO2}, that the $\tau$-Wigner is a generalized spectrogram if and only if $\tau=0$ or $\tau=1$.
The main result of this manuscript is Theorem \ref{thmFond} below, which characterizes the metaplectic Wigner distributions $W_\cA$ which are generalized spectrograms, that is for suitable windows $\phi,\psi$,
$$W_\cA(f,g)=V_\phi f\overline{V_\psi g},$$
in terms of the block decomposition of the symplectic matrix $\cA$ (to give a precise meaning to \eqref{spect}, when $\phi,\psi\in\cS'(\rd)$ we shall assume $f,g\in\cS(\rd)$). As a byproduct, we recapture the result that $\tau$-Wigner distributions are generalized spectrograms if and only if $\tau=0$ or $\tau=1$; also in Corollary \ref{corFPhiC} we compute explicitly $\theta_M$ in \eqref{thetaM} and \eqref{defBA}, even though $M$ is not invertible.  

Finally, given any $\cA$, $\cB$ in $Sp(2d,\bR)$, we introduce the 
   \textbf{$(\cA, \cB)$-metaplectic spectrogram} with windows $\phi$ and $\psi$  as
\[
{\Sp}^{\cA,\cB}_{\phi,\psi}(f,g):=W_\cA(f,\phi)\overline{W_\cB(g,\psi)}.
\]
For $\cA=A_{\tau_1}$ and $\cB=A_{\tau_2}$, $\tau_1,\tau_2\in\bR$, we recapture the parameterized two window spectrogram introduced in \cite{BdDO4}, whereas $\cA=\cB=A_{ST}$ in \eqref{AST} gives the classical one \eqref{Spectrogtam}.  In particular, we focus on the case $\cA=\cB$, which is called $\cA$-metaplectic spectrogram 	${\Sp}^\cA_{\phi,\psi}(f,g)$, and we answer the question when a metaplectic Wigner distribution can be represented in terms of $\cA$-metaplectic spectrogram.

{\bf Outline.}   Section $2$ contains the preliminaries required for this study, the definition of metaplectic Wigner distributions $W_\cA$ , their main properties, and related pseudodifferential calculus. It also shows that the Cohen distribution $\cB_{\mathbb{A}_M}$   defined in \eqref{Bamtheta} is a metaplectic Wigner distribution. Section $3$ is devoted to the boundedness properties of $W_\cA$ and related pseudodifferential operators on Lebesgue spaces. Section $4$ is the main focus of our study, containing the characterization of those metaplectic Wigner distributions that define generalized spectrograms. We extend the definition of generalized spectrograms to the wider class of metaplectic Wigner distributions in Section $5$. Finally, Section $6$ contains a thorough study of the $1$-dimensional setting: we shall show that the class of shift-invertible distributions has empty intersection with the class of generalized spectrograms considered in Section $4$ and the union of the two (disjoint) classes  gives exactly all the $\cA$-Wigner distributions of the Cohen's class.

\section{Preliminaries}\label{sec:preliminaries}
\textbf{Notation.} We denote %$t^2=t\cdot t$,  $t\in\rd$, and
$xy=x\cdot y$ (scalar product on $\Ren$). We denote by $\sch(\Ren)$ the Schwartz class and $\sch'(\Ren)$ is the space of tempered distributions. The brackets  $\la f,g\ra$ denote the extension to $\sch' (\Ren)\times\sch (\Ren)$ of the inner product $\la f,g\ra=\int f(t){\overline {g(t)}}dt$ on $L^2(\Ren)$ (conjugate-linear in the second component). We write a point in the phase space (or \tf\ space) as
$z=(x,\xi)\in\rdd$, and  the corresponding phase-space shift (\tfs )
acts on a function or distribution  as
\begin{equation}
\label{eq:kh25}
\pi (z)f(t) = e^{2\pi i \xi\cdot t} f(t-x), \, \quad t\in\rd.
\end{equation}
If $t\in\rd$, the Dirac delta distribution $\delta_t\in\cS'(\rd)$ is characterized by
\[
	\la\delta_t,\varphi\ra:=\overline{\varphi(t)}, \qquad \varphi\in\cS(\rd).
\]
%For $x\in\bR$, we write $\delta_0(x)=T_{-x}\delta_0$, where $T_{-x}f(t):=f(t+x)$ is the translation operator.
%We shall work with  lattices in the phase-space $\Lambda\subset \rdd$,  $\Lambda=A\zdd$,  with $A\in GL(2d,\R)$ and we will denote by $Q$ a fundamental domain containing the origin.  
%$\cC_0^\infty(\rdd)$ denotes the space of smooth functions with compact support. 
The notation $f\lesssim g$ means that there exists $C>0$ such that $ f(x)\leq Cg(x)$ holds for all $x$. The symbol $\lesssim_t$ is used when we stress that $C=C(t)$. If $ g\lesssim f\lesssim g$ or, equivalently, $ f \lesssim g\lesssim f$, we write $f\asymp g$. For a pair $E_1,E_2$ of Banach spaces,
$BL(E_1,E_2)$ is the space of linear and bounded operators $T: E_1\to E_2$.
If $M\in\bR^{d\times d}$, we use $\eig(M)$ to denote the set of the eigenvalues of $M$, whereas $\Sym(d,\bR)$ denotes the space of symmetric real $d\times d$ matrices and $\diag(\lambda_1,\ldots,\lambda_d)$ denotes the diagonal $d\times d$ matrix with $j$-th diagonal entry given by $\lambda_j$. Finally, $\det(M)=\prod_{\lambda\in\eig(M)\setminus\{0\}}\lambda$ denotes the pseudo-determinant of a given matrix $M\in\bR^{d\times d}$ and $M_{j\ast}$ denotes its $j$-th row ($j=1,\ldots,d$).
For two measurable functions $f,g:\rd\to\bC$, we set $f\otimes g(x,y):=f(x)g(y)$. If $X,Y$ are vector spaces, $X\otimes Y$ is the unique completion of $\text{span}\{x\otimes y : x\in X, y\in Y\}$. If $X(\rd)=L^2(\rd)$ or $\cS(\rd)$, the set $\text{span}\{f\otimes g:f,g\in X(\rd)\}$ is dense in $X(\rdd)$. Thus, for all $f,g\in\cS'(\rd)$, the operator $f\otimes g\in\cS'(\rdd)$ characterized by its action on $\varphi\otimes\psi\in\cS(\rdd)$ by
\[
	\la f\otimes g,\varphi\otimes\psi\ra = \la f,\varphi\ra\la g,\psi\ra
\]
extends uniquely to a tempered distribution of $\cS'(\rdd)$. The subspace $\text{span}\{f\otimes g: f,g\in\cS'(\rd)\}$ is dense in $\cS'(\rdd)$.

\subsection{Fourier transform}
The Fourier transform of $f\in\cS(\rd)$ is the function $\hat f\in\cS(\rd)$ defined as:
\[
	\hat f(\xi)=\int_{\rd} f(x)e^{-2\pi i\xi\cdot x}dx, \qquad \xi\in\rd.
\]
The operator $f\in\cS(\rd)\mapsto\hat f\in\cS(\rd)$ is a surjective isomorphism and, therefore, the relation
\[
	\langle \hat f,\hat\varphi\rangle=\langle f,\varphi\rangle, \qquad \varphi\in\cS(\rd).
\]
defines a distribution $\hat f\in\cS'(\rd)$ which is called \emph{Fourier transform} of $\hat f$. 

We denote with $\cF f:=\hat f$ the Fourier transform operator, which is also an isomorphism of $\cS'(\rd)$ onto itself, and a unitary operator on $L^2(\rd)$.

If $f\in\cS'(\rdd)$, we set $\cF_1f$ and $\cF_2f$ the partial Fourier transforms with respect to the first and to the second variables respectively, defined as:
\[
	\cF_1f(\xi,y)=\int_{\rd}f(x,y)e^{-2\pi i\xi\cdot x}dx, \qquad \cF_2f(x,\eta)=\int_{\rd}f(x,y)e^{-2\pi i\eta\cdot y}dy,
\]
for every $f\in\cS(\rdd)$, $x,y,\xi,\eta\in\rd$.

The symplectic Fourier transform $\cF_{\sigma}$ of a function $F$ on the phase space $\rdd$ is defined as 
\begin{equation}\label{SyFT}
	\cF_{\sigma}F(x,\o)=  \cF F (\o,-x).
\end{equation} 
\subsection{Time-frequency analysis tools}\label{subsec:23}
The \textit{short-time Fourier transform} of $f\in L^2(\rd)$ with respect to the window $g\in L^2(\rd)$ is the time-frequency representation defined in \eqref{stftdef}.
This definition extends to $(f,g)\in\cS'(\rd)\times\cS(\rd)$ by antilinear duality as $V_gf(x,\xi)=\langle f,\pi(x,\xi)g\rangle$.
The definition of the STFT extends to $\cS'(\rd)\times\cS'(\rd)$. In fact, $$V_gf=\cF_2\mathfrak{T}_{ST}(f\otimes\bar g), \qquad f,g\in\cS'(\rd),$$ where $\mathfrak{T}_{ST}F(x,y)=F(y,y-x)$ and $\cF_2$ is the partial Fourier transform with respect to the second coordinate, cf. Example \ref{es22} below. 

Recall the  $\tau$-Wigner distributions  ($\tau\in\bR$) defined in \eqref{tauWigner}.
The cases $\tau=0$ and $\tau=1$ are the so-called (cross-)\textit{Rihacek distribution}
\begin{equation}\label{RD}
W_0(f,g)(x,\xi)=f(x)\overline{\hat g(\xi)}e^{-2\pi i\xi\cdot x}, \quad \phas\in\rdd,
\end{equation}
 and (cross-)\textit{conjugate Rihacek distribution}
 \begin{equation}\label{CRD}
 W_1(f,g)(x,\xi)=\hat f(\xi)\overline{g(x)}e^{2\pi i\xi\cdot x}, \quad \phas\in\rdd.
 \end{equation}
 Observe that $$W_\tau (f,g)=\cF_2\mathfrak{T}_{\tau}(f\otimes\bar g), \qquad f,g\in\cS'(\rd),$$ where for any $F$ on $\rdd$, $$\mathfrak{T}_{\tau}F(x,y)=F(x+\tau y,x-(1-\tau)y),\quad x,y\in\rd.$$

\subsection{The symplectic group $Sp(d,\mathbb{R})$ and the metaplectic operators}\label{subsec:26}
	The \emph{symplectic group} $Sp(d,\bR)$ is the group of matrices $S\in\bR^{2d\times 2d}$ such that
	\begin{equation}\label{fundIdSymp}
	S^TJS=J,\end{equation}  the matrix $J$ being defined as
	\begin{equation}\label{defJ}
	J=\begin{pmatrix}
		0_{d\times d} & I_{d\times d}\\
		-I_{d\times d} & 0_{d\times d}
	\end{pmatrix},
\end{equation}
	where $I_{d\times d}\in\bR^{d\times d}$ is the identity matrix and $0_{d\times d}$ is the zero matrix of $\bR^{d\times d}$.
	
	If we write the block decomposition of $S\in Sp(d,\bR)$
	\begin{equation}\label{blocksA}
		S=\begin{pmatrix} A & B\\
		C & D\end{pmatrix}
	\end{equation}
	with $A,B,C,D\in\bR^{d\times d}$, \eqref{fundIdSymp} is equivalent to:
	\begin{align*}
			&(R1) \qquad \text{$A^TC=C^TA$},\\
			&(R2) \qquad \text{$B^TD=D^TB$},\\
			&(R3) \qquad \text{$A^TD-C^TB=I_{d\times d}$},
	\end{align*}
	and it can be proved that $\det(S)=1$. The inverse of a symplectic matrix $S$ with block decomposition \eqref{blocksA} is given in terms of its blocks by:
	\begin{equation}\label{invSymp}
		S^{-1}=\begin{pmatrix} D^T & -B^T\\
		-C^T & A^T\end{pmatrix}.
	\end{equation}
	%If $\cA\in Sp(d,\bR)$ and $\det(B)\neq0$, $\cA$ is called \textbf{free}. 
 %The matrix $S\in Sp(d,\bR)$ with block decomposition \eqref{blocksA} is called \emph{free} if $\det B\not=0$.
  
	For $E\in GL(d,\bR)$ and $C\in Sym(d,\bR)$, we define
	\begin{equation}\label{defDLVC}
		\cD_E:=\begin{pmatrix}
			E^{-1} & 0_{d\times d}\\
			0_{d\times d} & E^T
		\end{pmatrix} \qquad \text{and} \qquad V_C:=\begin{pmatrix}
			I_{d\times d} & 0\\ C & I_{d\times d}
		\end{pmatrix}.
	\end{equation}
	The matrix $J$ and those in the form $V_C$ ($C$ symmetric) and $\cD_E$ ($E$ invertible) generate the group $Sp(d,\bR)$.\\ 
	
	Despite being non-symplectic, the following matrix will appear ubiquitously throughout this work in relation to symplectic matrices:
\begin{equation}\label{defL}
	L=\begin{pmatrix}
		0_{d\times d} & I_{d\times d}\\
		I_{d\times d} & 0_{d\times d}
	\end{pmatrix}.
%\quad and \quad P=\begin{pmatrix}
%		0_{d	\times d} & I_{d\times d}\\
%		0_{d\times d} & 0_{d\times d}
%	\end{pmatrix}.
\end{equation}
	
	%Also, if $\cA\in Sp(d,\bR)$, there exist $\cA_1,\cA_2\in Sp(d,\bR)$ free such that $\cA=\cA_1\cA_2$. 

%Let $\rho$ be the Schr\"odinger representation of the Heisenberg group, that is $$\rho(x,\xi;\tau)=e^{2\pi i\tau}e^{-\pi i\xi\cdot x}\pi(x,\xi),$$ for all $x,\xi\in\rd$, $\tau\in\bR$. %We will use the following distribution property:  for all $f,g\in L^2(\rd)$, $z=(z_1,z_2),w=(w_1,w_2)\in\rdd$,
%%\[
%%	\rho(z;\tau)f\otimes\rho(w;\tau)g=e^{2\pi i\tau}\rho(z_1,w_1,z_2,w_2;\tau)(f\otimes g).
%%\]
%For all $S\in Sp(d,\bR)$, %$\rho_S(x,\xi;\tau):=\rho(S (x,\xi);\tau)$ defines another representation of the Heisenberg group that is equivalent to $\rho$, i.e., 
%there exists a unitary operator $\hat S:L^2(\rd)\to L^2(\rd)$ such that
%\begin{equation}\label{muAdef}
%	\hat S\rho(x,\xi;\tau)\hat S^{-1}=\rho(S(x,\xi);\tau), \qquad  x,\xi\in\rd, \ \tau\in\bR.
%\end{equation}
%This operator is not unique, but if $\hat S'$ is another unitary operator satisfying (\ref{muAdef}), then $\hat S'=c\hat S$, for some constant $c\in\bC$, $|c|=1$. The set $\{\hat S : S\in Sp(d,\bR)\}$ is a group under composition and it admits a subgroup that contains exactly two operators for each $S\in Sp(d,\bR)$. This subgroup is called \textbf{metaplectic group}, denoted by $Mp(d,\bR)$. It is 
The \emph{metaplectic group} $Mp(d,\bR)$ is a realization of the two-fold cover of $Sp(d,\bR)$ and the projection \begin{equation}\label{piMp}
	\pi^{Mp}:Mp(d,\bR)\to Sp(d,\bR)
\end{equation} is a group homomorphism with kernel $\ker(\pi^{Mp})=\{-id_{{L^2}},id_{{L^2}}\}$.

Throughout this work, if $\hat S\in Mp(d,\bR)$, the matrix $S$ (without the caret) will always be the unique symplectic matrix such that $\pi^{Mp}(\hat S)=S$.

\begin{proposition}{\cite[Proposition 4.27]{folland89}}\label{Folland427}
	Every operator $\hat S\in Mp(d,\bR)$ maps $\cS(\rd)$ isomorphically to $\cS(\rd)$ and it extends to an isomorphism on $\cS'(\rd)$.
\end{proposition}

Let $C\in\R^{d\times d}$ be symmetric and define 
\begin{equation}\label{chirp}
	\Phi_C(t)=e^{\pi i t\cdot Ct},\quad t\in\rd.
\end{equation}
If we add the assumption $C$ invertible, then its Fourier transform is given by
\begin{equation}\label{ft-chirp}
\widehat{\Phi_C}=c|\det(C)|^{-1/2}\,\Phi_{-C^{-1}},
\end{equation}
with $c$ phase factor.

\begin{example}\label{es22} The relationship up-to-a-sign between a metaplectic operator $\hat S\in Mp(d,\bR)$ and its projection $S$ is known for certain metaplectic operators $\hat S$. In fact, let $J$, $\cD_E$ and $V_C$ be defined as in (\ref{defJ}) and (\ref{defDLVC}), respectively. Then,
	\begin{enumerate}
		\item[\it (i)] The Fourier transform is a metaplectic operator and $\pi^{Mp}(\cF)=J$;
		\item[\it (ii)] Unitary linear change of variables are metaplectic operators: $$\mathfrak{T}_Ef:=|\det(E)|^{1/2}\,f(E\cdot)$$ and $\pi^{Mp}(\mathfrak{T}_E)=\cD_E$;
		\item[\it (iii)] The projection of the mapping $f\mapsto\Phi_C f$ is $V_C$;
		\item[\it (iv)] The projection of the operator $\psi_C =\cF \Phi_{-C} \cF^{-1}$ is $\pi^{Mp}(\psi_C)=V_C^T$;
		\item[\it (v)] If $\cF_2 $ is the Fourier transform with respect to the second variables, then $\pi^{Mp}(\cF_2)=\cA_{FT2}$, where $\cA_{FT2}\in Sp(2d,\bR)$ is the $4d\times4d$ matrix with block decomposition
		\begin{equation}\label{AFT2}
		\cA_{FT2}:=\begin{pmatrix}
			I_{d\times d} & 0_{d\times d} & 0_{d\times d} & 0_{d\times d}\\
			0_{d\times d} & 0_{d\times d} & 0_{d\times d} & I_{d\times d} \\
			0_{d\times d} & 0_{d\times d} & I_{d\times d} & 0_{d\times d}\\
			0_{d\times d} & -I_{d\times d} & 0_{d\times d} & 0_{d\times d}
		\end{pmatrix}.
		\end{equation}
%		\item[\it (vi)] Assume $S=\pi^{Mp}(\hat S)\in Sp(d,\bR)$ has block decomposition (\ref{blocksA}). Then,\\
%		if $S$ is free:
%	\begin{equation}\label{free-int-rep}
%		\hat S f(x)=(\det(B))^{-1/2}\Phi_{-DB^{-1}}(x)\int_{\rd}e^{2\pi i y\cdot B^{-1}x}\Phi_{-B^{-1}A}(y) f(y)dy.
%	\end{equation}
%		if $\det A\not=0$,
%			\begin{equation}\label{Anonsing-int-rep}
%		\hat S f(x)=(\det(A))^{-1/2}\Phi_{-CA^{-1}}(x)\int_{\rd}e^{2\pi i \xi\cdot A^{-1}x}\Phi_{-A^{-1}B}(\xi) \hf(\xi)d\xi.
%			\end{equation}
	\end{enumerate}

\end{example}

\subsection{Metaplectic Wigner distributions}\label{2.4}
Let $\hat\cA\in Mp(2d,\bR)$ and $\cA=\pi^{Mp}(\hat\cA)$. The \textit{metaplectic Wigner distribution} associated to $\hat\cA$ is defined as
\[
	W_\cA(f,g)=\hat\cA(f\otimes\bar g),\quad f,g\in L^2(\rd).
\]
On the left hand-side we omit the caret, by abuse of notation. We set $W_\cA f=W_\cA(f,f)$. Note that $\cA\in Sp(2d,\bR)$ means that $\cA$ is a $4d\times 4d$ symplectic matrix, for which we assume a block decomposition of the type \eqref{blockDecA}.

We recall the following continuity properties.
\begin{proposition}\label{prop25} If $X(\rd)$ denotes one of the following spaces: $L^2(\rd), \cS(\rd), \cS'(\rd)$, and  $W_\cA$ is a metaplectic Wigner distribution, then 
	 $W_\cA:X(\rd)\times X(\rd)\to X(\rdd)$ is bounded.
\end{proposition}
Since metaplectic operators are unitary, for all $f_1,f_2,g_1,g_2\in L^2(\rd)$,
\begin{equation}\label{Moyal}
	\la W_\cA(f_1,f_2),W_\cA(g_1,g_2)\ra = \la f_1,g_1\ra \overline{\la f_2,g_2\ra}.
\end{equation}

%The projection of a metaplectic operator $\hat\cA\in Mp(2d,\bR)$ is a symplectic matrix $\cA\in Sp(2d,\bR)$ with block decomposition \eqref{blockDecA}.
\begin{proposition}\label{prop-conn}
	The  Cohen distribution  $\mathcal{B}_{\mathbb{A}_M}$ defined in \eqref{Bamtheta},  having the Cohen kernel $\theta_M$ in \eqref{thetaM}, is a metaplectic Wigner distribution in the Cohen class of the type \eqref{covMWDE}, with symmetric matrix $\mathcal{M}_\cA$ in \eqref{MA}.
\end{proposition}
\begin{proof}
The connection between the  Fourier transform $\cF$ and the symplectic Fourier transform in $\cF_\sigma$ in \eqref{SyFT} gives in particular that $\cF_\sigma F\phas=\cF^{-1} F(-\xi,x)$, for any $F\in L^2(\rdd)$. For $\zeta=\phas$, it is then straightforward to check that  $\cF_\sigma\chi_M =\cF^{-1}(e^{-i\pi \mathcal{M}_A\zeta\cdot\zeta})$, and the thesis follows.
\end{proof}

\subsection{Metaplectic pseudodifferential operators}
In what follows we recall the definition of the pseudodifferential operators associated with the metaplectic Wigner distribution $W_\cA$.
\begin{definition}\label{defMetaplPsiDo}
		Consider  $a\in\mathcal{S}'(\mathbb{R}^{2d})$. The \textit{metaplectic pseudodifferential operator} with \textit{symbol} $a$ and symplectic matrix $\cA\in Sp(2d,\bR)$ is the operator $Op_\mathcal{A}(a):\mathcal{S}(\mathbb{R}^d)\to\mathcal{S}'(\mathbb{R}^{d})$ defined by
		\[		\langle Op_\mathcal{A}(a)f,g\rangle=\langle a,W_\mathcal{A}({g},f)\rangle, \quad f,g\in\mathcal{S}(\mathbb{R}^d).	\]
	\end{definition}
This operator is well defined by Proposition \ref{prop25}. The metaplectic Wigner distribution $W_\cA$ plays the role of a \emph{quantization}. In particular, choosing $\cA$ properly, we recapture the most well-know quantizations. For $\cA=\cA_\tau$ in \eqref{Atau},  we recover the $\tau$-operators, cf.  \cite{CR2021}: for $f\in\cS(\rd)$,
\begin{equation}\label{optau}
	Op_\tau(a)f(x):= Op_{\cA_\tau}(a)f(x)=\intrdd e^{2\pi i(x-y)\xi}a((1-\tau)x+\tau y,\xi)f(y)dyd\xi.
\end{equation}
For $\tau=1/2$ we obtain the Weyl operator $Op_w(a)$
\begin{equation}\label{weyl}
	Op_w(a)f(x)=Op_{1/2}(a)f(x)=\intrdd e^{2\pi i(x-y)\xi}a\left(\frac{x+y}2,\xi\right)f(y)dyd\xi,
\end{equation}
whereas $\tau=0$ gives the Kohn-Nirenberg  operator
\begin{equation}\label{optau0}
	Op_0(a)f(x)=\intrdd e^{2\pi i(x-y)\xi}a(x,\xi)f(y)dyd\xi.
\end{equation}
If $\cA=I_{4d\times4d}$, the $4d\times4d$ identity matrix, then $W_\cA(f,g)=f\otimes\bar{g}$ and $Op_\cA(a)$ is the operator with Schwartz kernel $a$, see for example \cite[Section 14.4]{book}. 
\section{Boundedness of shift-invertible $W_\cA$ and related $Op_\cA$ on $L^p$ spaces}\label{sec:BSI}
The modern methods of the time-frequency analysis involve a large use of function spaces. In particular, besides Schwartz functions and distributions, the classical Lebesgue spaces $L^p$, $1\leq p\leq\infty$, give a precise framework to several formal results. Another relevant setting is provided by the modulation spaces of Feichtinger \cite{feichtinger-modulation}. When considering boundedness properties, a relevant sub-class of the metaplectic Wigner distributions is identified by the so-called shift-invertibility property, cf. \cite{CR2022}, granting good continuity results. Though our main results, Theorem \ref{thmFond} and subsequent corollaries, do not involve such function spaces, it will be natural to consider spectrograms in a functional frame, see Remark \ref{rem6.2} at the end of the paper. As a possible preparation to a precise study of spectrograms in this perspective, in the present section we analyze the boundedness properties of $W_\cA$ and the related metaplectic operator $Op_\cA(a)$ on $L^p(\rd)$ spaces, generalizing the results for $\tau$-Wigner distributions, shown in \cite{BdDO2}.
We use the following issues, cf. \cite[Proposition 2.1]{BdDO1} and \cite[Proposition 6.3]{BdDO2}.
\begin{proposition}\label{BdDO11}
	Let $E,E_1,E_2$ be Banach spaces and $E_2$ be reflexive. \\
	(i) Let $\varphi:E_2^\ast\times E_1\to E^\ast$ be a sesquilinear bounded mapping. Then, there exists a linear bounded mapping $E\ni a \mapsto T_a\in BL(E_1,E_2)$ such that
	\begin{equation}\label{BdDO-eq1}
		(v,T_au)=(\varphi_{v,u},a) \qquad \forall v\in E_2^\ast.
	\end{equation}
	(ii) Let $E\ni a\mapsto T_a\in BL(E_1,E_2)$ be a linear bounded mapping. Then, (\ref{BdDO-eq1}) defines a sesquilinear bounded mapping $\varphi:E_2^\ast\times E_1\to E^\ast$.
\end{proposition}

Let $\cA\in Sp(2d,\bR)$. Applying Proposition \ref{BdDO11} with $T_a=Op_\cA(a)$ and $\varphi=W_\cA$, we obtain:
\begin{proposition}\label{CGR-prop1}
	Let $E,E_1,E_2$ be Banach spaces and $E_2$ be reflexive, let $W_\cA$ be the metaplectic Wigner distribution associated to $\hat\cA\in Mp(2d,\bR)$ and $\pi^{Mp}(\hat\cA)=\cA\in Sp(2d,\bR)$ be the projection of $\hat\cA$ onto $Sp(2d,\bR)$. The following statements are equivalent:\\
	(i) $ E\ni a\mapsto Op_\cA(a)\in BL(E_1,E_2)$ is a bounded mapping.\\
	(ii) $W_\cA:E_2^\ast\times E_1\mapsto E^\ast$ is continuous.
\end{proposition}

We also need the following estimates for the short-time Fourier transform, let us refer again to \cite{BdDO1,BdDO2}.

\begin{proposition}\label{BdDO21}
	The STFT $V:(f,g)\in L^{p'}(\rd)\times L^p(\rd)\mapsto V_gf\in L^q(\rdd)$ is bounded if and only if $q\geq2$ and $q'\leq p\leq q$, with
	\[
		\norm{V_gf}_q\leq \norm{f}_{p'}\norm{g}_p.
	\]
\end{proposition}

Let us recall the original definition of shift-invertibility.

\begin{definition}
	A metaplectic Wigner distribution $W_\cA$ is \emph{shift-invertible} if
	\[
		|W_\cA(\pi(w)f,g)|=|T_{E_\cA w}W_\cA(f,g)|, \qquad f,g\in L^2(\rd), \quad w\in\rdd,
	\]
	for some $E_\cA\in GL(2d,\bR)$, with:
	\[
		T_{E_\cA w}W_\cA(f,g)(z)=W_\cA(f,g)(z-E_\cA w), \qquad z,w\in\rdd.
	\]
\end{definition}

The definition was recently rephrased in different equivalent forms. In what follows we use the characterization of shift-invertible distributions provided in \cite[Corollary 3.3]{CG-Excursus}. Namely:
\begin{theorem}\label{CG-thm1}
A metaplectic Wigner distribution $W_\cA$ is shift-invertible if and only if, for every $f\in\cS'(\rd)$, $g\in\cS(\rd)$,
\begin{equation}\label{S-inv}
	W_\cA(f,g)(z)=|\det(E)|^{1/2}\Phi_{C}(Ez)V_{\hat Sg}f(Ez), \quad  \ z\in\rdd,
\end{equation}
for some $E\in GL(2d,\bR)$, $C\in \Sym(2d,\bR)$ and $\hat S\in Mp(d,\bR)$. 
\end{theorem}

In the sequel we shall take \eqref{S-inv} as the very definition of shift-invertibility. Note that \eqref{S-inv} is obviously satisfied by the STFT with $S=E=I_{2d\times 2d}$, $C=0_{2d\times 2d}$. Also the Wigner transform in \eqref{CWD} and $W_\tau$ in \eqref{tauWigner} satisfy \eqref{S-inv}, apart from the cases $\tau=0$ or $\tau=1$, cf. \cite[Lemma 6.2]{BdDO2}. As for the distributions $\mathcal{B}_{\mathbb{A}}$ in \eqref{BAe}, \eqref{AM}, and \eqref{Bamtheta}, they are shift-invertible if and only if $M+(1/2) I_{d\times d}$ and $M-(1/2) I_{d\times d}$ are invertible, i.e., the matrix \eqref{AM} is right-invertible, according to \cite[Theorem 2]{BCGT2020}.

We can also write shift-invertible $W_\cA$ in terms of the classical Wigner distribution.

\begin{lemma}\label{CGR-lemma1}
	A metaplectic Wigner distribution $W_\cA$ is shift-invertible if and only if, for every $f\in\cS'(\rd)$, $g\in\cS(\rd)$,
	\begin{equation}\label{charWAW}
		W_\cA(f,g)(z)=|\det(E')|^{1/2}\Phi_{C'}(E'z)W(f,\hat S'g)(E'z), \ z\in\rdd,
	\end{equation}
	for some $E'\in GL(2d,\bR)$, $C'\in \Sym(2d,\bR)$ and $\hat S'\in Mp(d,\bR)$. Moreover, if (\ref{S-inv}) is the expression of $W_\cA$ in terms of the STFT, then:
	\[
		E'=\frac{1}{2}E, \qquad C'=4C-2E^{-T}LE^{-1}, \quad and \quad \hat S'=\mathcal{I}\hat S,
	\]
	where $L$ is defined as in (\ref{defL}) and $\mathcal{I}g(t):=g(-t)$. Consequently, if $a\in\cS'(\rdd)$, the metaplectic operator $Op_\cA(a)$ can be rephrased as  the Weyl operator 
	\begin{equation}\label{CGR-eq3}
		Op_\cA(a)f(x)=Op_w(\Phi_{-C'}\cdot\mathfrak{T}_{(E')^{-1}}a)\circ \hat S'f(x),\quad f\in\cS(\rd),
	\end{equation}
where the change of variables $\mathfrak{T}_{(E')^{-1}}$ is defined in Example \ref{es22} $(ii)$.
\end{lemma}
\begin{proof}
	From (\ref{S-inv}) and using the relationship between the STFT and the classical Wigner distribution \cite[Lemma 1.3.5]{Elena-book}
	\[
		V_gf(x,\xi)=2^{-d}e^{-i\pi x\cdot\xi}W(f,\mathcal{I}g)\left(\frac{x}{2},\frac{\xi}{2}\right), \qquad x,\xi\in\rd,
	\]
	where $\mathcal{I}g(t):=g(-t)$, we have:
	\begin{align*}
	W_\cA(f,g)(x,\xi)&=|\det(E)|^{1/2}\Phi_{C}(E (x,\xi))\cdot 2^{-d}e^{-i\pi x\cdot\xi}W(f,\mathcal{I}\hat Sg)\left(\frac{1}{2}E(x,\xi)\right).
	%&=|\det(\frac{1}{2}E)|^{1/2}\Phi_{2C-L}(\frac{1}{2}E (x,\xi))W(f,\mathcal{I}\widehat{S}g)(\frac{1}{2}E(x,\xi))),
	\end{align*}
	where $L$ is defined as in (\ref{defL}). Observe that, if $z=(x,\xi)$, 
	\[
		\Phi_{C}(Ez)=e^{i\pi C(Ez)\cdot(Ez)}=e^{i\pi 4C (\frac{1}{2}Ez)\cdot(\frac{1}{2}Ez)}=\Phi_{4C}\left(\frac{1}{2}Ez\right),
	\]
	whereas:
	\begin{align*}
		e^{-i\pi x\cdot \xi}&=e^{-i\pi \frac{L}{2}z\cdot z}=e^{-i\pi[\frac{L}{2}(2E^{-1}\frac{1}{2}Ez)]\cdot(2E^{-1}\frac{1}{2}Ez)}=e^{i\pi [(-2E^{-T}LE^{-1})(\frac{1}{2}Ez)]\cdot(\frac{1}{2}Ez)}\\
		&=\Phi_{-2E^{-T}LE^{-1}}\left(\frac{1}{2}Ez\right).
	\end{align*}
	Since $E^{-T}LE^{-1}\in \Sym(2d,\bR)$, it follows that:
	\begin{align*}
	W_\cA(f,g)(x,\xi)&=\left|\det\left(\frac{1}{2}E\right)\right|^{1/2}\Phi_{4C-2E^{-T}LE^{-1}}\left(\frac{1}{2}E(x,\xi)\right)W(f,\mathcal{I}\hat Sg)\left(\frac{1}{2}E(x,\xi)\right).
	\end{align*}
	The characterization and (\ref{charWAW}) follow by Theorem \ref{CG-thm1}, setting $E'=\frac{1}{2}E$, $C'=4C-2E^{-T}LE^{-1}$ and $\hat S'=\mathcal{I}\hat S$, and observing that all the mappings $ GL(2d,\bR)\ni E\mapsto E'\in GL(2d,\bR)$, $\Sym(2d,\bR)\ni C \mapsto C'\in \Sym(2d,\bR)$ and $ Mp(d,\bR)\ni \hat S\mapsto \hat S'\in Mp(d,\bR)$ are bijective.
	
	To prove (\ref{CGR-eq3}), let $f\in\cS(\rd)$ and $a\in\cS'(\rdd)$. Then, for every $g\in\cS(\rd)$,
	\begin{align*}
		\langle Op_\cA(a)f,g \rangle & = \langle a,W_\cA(g,f)\rangle \\
		&= \langle a, |\det(E')|^{1/2} (\Phi_{C'}W(g,\hat S' f))\circ E' \rangle\\
		& = \langle a, \mathfrak{T}_{E'}(\Phi_{C'}W(g,\hat S' f))\rangle\\
		&=\langle \Phi_{-C'}\cdot \mathfrak{T}_{(E')^{-1}}a, W(g,\hat S'f) \rangle\\
		&=\langle Op_w(\Phi_{-C'}\cdot\mathfrak{T}_{(E')^{-1}}a)(\hat S'f),g\rangle
	\end{align*}
	and the assertion follows.
\end{proof}

The boundedness of shift-invertible metaplectic Wigner distributions follows then by the previous characterization and by the result for the STFT.

\begin{proposition}\label{CGR-cor1}
	Let $W_\cA$ be a shift-invertible metaplectic Wigner distribution as in (\ref{S-inv}). Assume that for every $1\leq p\leq \infty$, $\hat S:L^p(\rd)\to L^p(\rd)$ is a topological isomorphism with $\norm{\hat S g}_{p}\leq C_{S,p}\norm{g}_p$ and $\norm{\hat S^{-1}g}_p\leq C_{S,p}'\norm{g}_p$. Then, $W_\cA:(f,g)\in L^{p'}(\rd)\times L^p(\rd)\to W_\cA(f,g)\in L^q(\rdd)$ is bounded if and only if $q\geq 2$ and $q'\leq p\leq q$.
\end{proposition}
\begin{proof}
	For $f\in\cS'(\rd)$, $g\in\cS(\rd)$, Theorem \ref{CG-thm1} gives
	\[	
		W_\cA(f,g)(z)=|\det(E)|^{1/2}\Phi_C(Ez)V_{\hat Sg}f(Ez), \ z\in\rdd.
	\]
 The characterization follows by Proposition \ref{BdDO21}, using that $L^q(\rdd)\ni F \mapsto F(E\cdot)\in L^q(\rdd)$ and $ L^{p}(\rd)\ni g \mapsto \hat S g\in L^{p}(\rd)$ are topological isomorphisms. Namely, the sufficiency is a consequence of the continuity of $\hat S:L^p(\rd)\to L^p(\rd)$. For the necessity we use the pattern for the STFT used in \cite[Proposition 3.2]{BdDO1},  and   the continuity of $\hat S^{-1}:L^p(\rd)\to L^p(\rd)$. Precisely, define $\varphi(x):=e^{-\pi|x|^2}$ and $\varphi_\lambda(x):=e^{-\pi\lambda|x|^2}$, for every $\lambda>0$. Since 
 \[
 \norm{W_\cA(h,g)}_q^q=|\det(E)|^{q/2-1}\norm{V_{\hat S g}h}_q^q,
 \] 
 and choosing $h=\varphi_\lambda$ and $g=\hat S^{-1}\varphi$ we obtain
	\[
		\frac{\norm{W_\cA(\varphi_\lambda,\hat S^{-1}\varphi)}_q}{\norm{\varphi_\lambda}_{p'}\norm{\hat S^{-1}\varphi}_p}=\frac{|\det(E)|^{1/2-1/q}\norm{V_\varphi\varphi_\lambda}_{q}}{C_{S,p}'\norm{\varphi_\lambda}_{p'}\norm{\varphi}_p}.
	\]
	It is then enough to show that the member on the right hand-side is unbounded when $q\geq 2$ or $p$ does not satisfy the condition $q'\leq p\leq q$. This is shown in  \cite[Proposition 3.2]{BdDO1}.
\end{proof}
\begin{remark}
Concerning the assumption on $\hat S$, observe that not every metaplectic operator is bounded on $L^p(\rd)$, $1\leq p\leq \infty$. In particular, this is the case of the Fourier transform and the Fourier multipliers, cf. Example \ref{es22} $(i)$ and $(iv)$. 
	Instances of metaplectic operators that are isomorphisms on $L^p(\rd)$ are the dilation operator and the multiplication by a chirp, shown in Example \ref{es22} $(ii)$ and $(iii)$, respectively.
	
	This is the case for $W_\tau$ with $\tau\not=0$ and $\tau\not=1$, and for the distribution $\mathcal{B}_{\mathbb A}$ in \eqref{BAe}, \eqref{AM}, \eqref{Bamtheta}, under the assumption of shift-invertibility, see \cite[Lemma 6.2]{BdDO2} and \cite[Theorem 2]{BCGT2020}; for them $\hat S$ is a dilation, i.e., a linear change of variables.
	
	When $W_\cA$ is not shift-invertible, boundedness as in Proposition \ref{CGR-cor1} fails, though weaker results of continuity are possible, see \cite[Proposition 6.4 $(ii)$, $(iii)$]{BdDO2} for the Rihacek distribution and \cite[Theorem 5]{BCGT2020} for $\mathcal{B}_{\mathbb{A}}$ non-right-regular. 
\end{remark}
\begin{theorem}\label{CGR-thm1}	
	Let $q\geq1$ and $W_\cA$ be a shift-invertible distribution as in (\ref{charWAW}).
%	with
%	\[
%	W_\cA(f,g)(z)=|\det(E')|^{1/2}\Phi_{C'}(E'z)W(f,\hat S'g)(E'z),
%	\]
%	 $f\in\cS'(\rd)$, $g\in\cS(\rd)$, $z\in\rdd$.
	Assume that, for every $1\leq p\leq\infty$, $\hat S':L^p(\rd)\to L^p(\rd)$ is a topological isomorphism. Then, the mapping $a\in L^q(\rdd)\mapsto Op_\cA(a)\in BL(L^p(\rd))$ is bounded if and only if $q\leq 2$ and $q\leq p\leq q'$.
\end{theorem}
\begin{proof}
	We give two alternative proofs, motivated by different perspectives: the functional analytic approach from Proposition \ref{CGR-prop1} and, respectively, the equivalent definition of shift-invertibility from Lemma \ref{CGR-lemma1}.

	(i) We use Proposition \ref{CGR-prop1} with $E=L^q(\rdd)$, $E_1=E_2=L^p(\rd)$. From Proposition \ref{CGR-cor1} we know that the assumptions of Theorem \ref{CGR-thm1} are necessary and sufficient for $(ii)$ in Proposition \ref{CGR-prop1}, namely boundedness of the mapping 
	$$W_\cA: L^p(\rd)\times L^{p'}(\rd)\to L^{q'}(\rdd)$$
	(with respect to the notation of Proposition \ref{CGR-prop1} we consider the dual index $q'\geq 2$). We obtain the equivalence with $(i)$ in Proposition \ref{CGR-prop1}, which reads $Op_\cA(a)\in BL(L^p(\rd))$.\par
	(ii) We refer directly to Lemma \ref{CGR-lemma1}. In fact, it was proved in \cite[Theorem 3.3]{BdDO1} that the Weyl quantization $a\in L^q(\rdd)\mapsto Op_w(a)\in BL(L^p(\rd))$ is continuous if and only if $q\leq 2$ and $q\leq p\leq q'$. By Lemma \ref{CGR-lemma1}  the relation between a metaplectic and a Weyl operator is 
	\[
	Op_\cA(a)=Op_w(\Phi_{-{C'}}\cdot\mathfrak{T}_{(E')^{-1}}a)\circ \hat S'.
	\]
Setting $1/\infty=0$, we have $$\norm{\Phi_{-{C'}}\cdot\mathfrak{T}_{(E')^{-1}}a}_q=|\det(E')|^{1/q}\cdot\norm{a}_q,$$ for every $1\leq q\leq\infty$. Since  $\hat S':L^p(\rd)\to L^p(\rd)$ is a topological isomorphism for every $1\leq p\leq \infty$, the assertion follows from the characterization of the  boundedness of the Weyl quantization on Lebesgue spaces in \cite[Theorem 3.3]{BdDO1}.
\end{proof}
\begin{remark}\label{remark3.9} From Theorem \ref{CGR-thm1} we recapture the results in \cite[Theorem 6.6 $(i)$]{BdDO2} and \cite[Theorem 15]{BCGT2020}.
	
	Proposition \ref{CGR-prop1} can be applied to other metaplectic quantizations $Op_\cA(a)$. For example, if $\hat{\cA}$ is the identity or a dilation, or a multiplication by a chirp, then for $1<q<\infty$ we have
	$$W_\cA: L^q(\rd)\times L^q(\rd)\to L^q(\rdd),$$
	hence from Proposition \ref{CGR-prop1}
	$$L^{q'}(\rdd) \ni a\to Op_\cA(a)\in BL(L^{q}(\rd),L^{q'}(\rd)),$$
as we may obtain by direct estimates on Schwartz kernels. These examples being outside the time-frequency perspective, we omit further results in this direction.
\end{remark}
\section{Metaplectic Generalized Spectrograms}
The issue that we approach in this section is the characterization of those metaplectic Wigner distributions that define generalized spectrograms. Recall that this implies they belong to the Cohen's class.
 By \cite[Theorem 2.11]{CR2022}, a metaplectic Wigner distribution $W_\cA$ belongs to the Cohen's class if and only if
\begin{equation}\label{covMWD}
	W_\cA(f,g)= W(f,g)\ast\cF^{-1}(e^{-i\pi \cM_\cA\zeta\cdot\zeta}),
\end{equation}
where, given $\cA=\pi^{Mp}(\hat \cA)$ having block decomposition (\ref{blockDecA}), $\cM_\cA$ is defined as in (\ref{defBA}).

We need the following notion of invertibility. Let $C\in\bR^{d\times d}$ be symmetric with eigenvalues $\lambda_1,\ldots,\lambda_d$ and $\Sigma\in\bR^{d\times d}$ be orthogonal such that
\begin{equation}\label{decompC}
	C=\Sigma^T\Delta_C\Sigma,
\end{equation}
where $\Delta_C=\diag(\lambda_1,\ldots,\lambda_d)$ is the diagonal matrix with diagonal entries given by the eigenvalues of $C$. In the following, we will always assume that the $j$-th column of $\Sigma$ corresponds to $\lambda_j$. 

\begin{definition}\label{defGenInv} Under the notation above, we define the \textbf{generalized inverse} of $C$, named  $C^{-}$, the matrix
\begin{equation}
	C^{-}:=\Sigma^T\Delta_C^{-}\Sigma, 
\end{equation}
where $\Delta_C^{-}$ is the diagonal matrix $d\times d$ whose $j$-th diagonal entry is given by:
\[
	(\Delta_C^{-})_{jj}=\begin{cases}
		\lambda_j^{-1} & \text{if $\lambda_j\neq0$},\\
		0 & \text{otherwise},
	\end{cases}
\]
$j=1,\ldots,d$.
\end{definition}

%\begin{remark}\label{remFond}
%	Let $x\in\bR^d$ and $\Delta=\diag\{\lambda_1,\ldots,\lambda_d\}$ be any diagonal matrix. Then, 
%	\begin{align*}
%		\Delta x\cdot y&=\begin{pmatrix}
%			x_1 & \ldots & x_d
%		\end{pmatrix}\begin{pmatrix}
%			\lambda_1 & \ldots & 0\\
%			\vdots & \ddots & \vdots\\
%			0 & \ldots & \lambda_d
%		\end{pmatrix}\begin{pmatrix}
%			y_1\\
%			\vdots\\
%			y_d
%		\end{pmatrix}=\begin{pmatrix}
%			\lambda_1x_1 & \ldots & \lambda_dx_d
%		\end{pmatrix}
%		\begin{pmatrix}
%			y_1\\
%			\vdots\\
%			y_d
%		\end{pmatrix}=\sum_{j=1}^d\lambda_jx_jy_j
%	\end{align*}
%\end{remark}

Let us also introduce the following simplifying notation:
\begin{definition}\label{defDelta}
	Let $\mathfrak{Z}\subseteq\{1,\ldots,d\}$. We define $\delta_{\mathfrak{Z}}\in\cS'(\rd)$ as
	\[
		\delta_{\mathfrak{Z}}(\xi):=\begin{cases}
		\bigotimes_{j=1}^d\psi_j(\xi_j) & \text{if $\mathfrak{Z}\neq\emptyset$},\\
		1 & \text{if $\mathfrak{Z}=\emptyset$},
		\end{cases}
	\]
	where
	\[
		\psi_j=\begin{cases}
			\delta_0 & \text{if $j\in\mathfrak{Z}$},\\
			1 & \text{otherwise},
		\end{cases}
	\]
	where $\delta_0$ is the Dirac's delta distribution defined at the beginning of Section $2$. In this case, we write $\delta_{\mathfrak{Z}}(\xi)=\prod_{j\in\mathfrak{Z}}\delta_0(\xi_j)$.
\end{definition}

In general, given a partition $\{1,\ldots,d\}=\mathfrak{Z}_1\sqcup\mathfrak{Z}_2$, with $\#(\mathfrak{Z}_1)=r$, $\mathfrak{Z}_1=\{1\leq j_1<\ldots<j_r\leq d\}$, $\mathfrak{Z}_2=\{1\leq k_1<\ldots<k_{d-r}\leq d\}$, and given $f_{j_1},\ldots,f_{j_r},g_{k_1},\ldots,g_{k_{d-r}}\in\mathcal{S}'(\bR)$ we write $$\bigotimes_{j=1}^d h_j\in\cS'(\rd)=\prod_{j\in\mathfrak{Z}_1}f_j\prod_{k\in\mathfrak{Z}_2}g_k$$ to denote the tempered distribution with
\[
h_j=\begin{cases}
	f_j & \text{if $j\in\mathfrak{Z}_1$},\\
	g_j & \text{if $j\in\mathfrak{Z}_2$}.
\end{cases}
\]

We need the following a fact, that we will use in the following.

\begin{remark}
	Let $A_{1/2}$ be the symplectic matrix  in (\ref{Atau}) with $\tau=1/2$. Then $W(f,g)=\widehat{A}_{1/2}(f\otimes\bar g)=\cF_2\mathfrak{T}_W(f\otimes\bar g)$, where $\mathfrak{T}_W$ is the change of variables:
	\[
		\mathfrak{T}_W F(x,y)=F\left(x+\frac{y}{2},x-\frac{y}{2}\right).
	\]
	The inverse of $\mathfrak{T}_W$ is given by:
	\begin{equation}\label{TWm1}
		\mathfrak{T}_W^{-1}F(x,y)=F\left(\frac{x+y}{2}x,x-y\right).
	\end{equation}
%	A simple computation shows that the inverse of $\widehat{A_{1/2}}$ is given by:
%	\begin{equation}\label{invW}
%		\widehat{A_{1/2}}^{-1}F(x,y)=\int_{\rd}F\left(\frac{x+y}{2},\eta\right)e^{2\pi i(x-y)\cdot\eta}d\eta.
%	\end{equation}
\end{remark}

\begin{lemma}\label{lemmaFPhiC}
	Let $C\in\bR^{d\times d}$ be symmetric, $\Sigma$ orthogonal  with decomposition (\ref{decompC}), and $\eig(C)=\{\lambda_1,\ldots,\lambda_d\}$. Then, 
	if $\mathfrak{Z}:=\{j:\lambda_j=0\}$,
	\begin{equation}\label{FPhiC2}
		\cF\Phi_C(\xi)=s|\det(C)|^{-\frac12}\Phi_{-C^{-}}(\xi)\delta_{\mathfrak{Z}}(\Sigma\xi),
	\end{equation}
	where $s$ is a phase factor ($s\in\bC$ such that $|s|=1$) defined as in the subsequent proof, $\delta_{\mathfrak{Z}}$ is defined as in Definition \ref{defDelta}, and $\det(C)$ denotes the pseudo-determinant of $C$.
	
	If $\mathfrak{Z}=\emptyset$, i.e., $C$ is invertible, then $\delta_{\mathfrak{Z}}$ does not appear, $C^-=C^{-1}$ and we recapture the standard formula, cf. \cite{folland89}.
\end{lemma}
\begin{proof}
Let $\eig(C)=\{\lambda_1,\ldots,\lambda_d\}$. Let us decompose $C=\Sigma^T\Delta\Sigma$, where $\Sigma\in\bR^{d\times d}$ is orthogonal and $\Delta=\diag(\lambda_1,\ldots,\lambda_d)$ is the $d\times d$ diagonal matrix whose diagonal entries are the eigenvalues of $C$. Then, using the change of variables $\Sigma x=y$ ($x=\Sigma^T y$, $dx=dy$), 
	\begin{align*}
		\cF(\Phi_C)(\xi)&=\int_{\rd}e^{i\pi Cx\cdot x}e^{-2\pi i\xi\cdot x}dx=\int_{\rd}e^{i\pi [\Delta (\Sigma x)]\cdot(\Sigma x)}e^{-2\pi i\xi\cdot x}dx\\
		&=\int_{\rd}e^{i\pi \Delta y\cdot y}e^{-2\pi i\Sigma^T y\cdot \xi}dy
		=\int_{\rd}e^{i\pi\Delta y\cdot y}e^{-2\pi iy\cdot (\Sigma \xi)}dy\\
		&=[\cF(e^{i\pi \Delta y\cdot y})](\Sigma \xi)=[\cF(\otimes_{j=1}^de^{i\pi \lambda_jy_j^2})](\Sigma\xi)\\
		&=\left[\bigotimes_{j=1}^d\cF(e^{i\pi\lambda_j y_j^2})\right](\Sigma \xi).
	\end{align*}
	For $\alpha\in\bR$, 
	\[
		\cF(e^{i\pi \alpha x^2})=\begin{cases}
			s|\alpha|^{-1/2}\Phi_{-\alpha^{-1}} & \text{if $\alpha\neq0$},\\
			\delta_0 & \text{otherwise},
		\end{cases}
	\]
	where the phase factor $s$ is determined by the sign $\alpha$. Therefore, formula 
	\begin{equation}\label{FPhiC}
		\cF\Phi_C(\xi)=s|\det(C)|^{-1/2}\left(\bigotimes_{j=1}^d\varphi_j\right)(\Sigma \xi),
	\end{equation}
	follows, where 
	\[
		\varphi_j=\begin{cases}
			\Phi_{-\lambda_j^{-1}} & \text{if $\lambda_j\neq0$}\\
			\delta_0 & \text{otherwise},
		\end{cases}
	\]
	and the phase factor $s$ depends on the signature of the non-zero eigenvalues of $C$.
	To prove that (\ref{FPhiC}) coincides with (\ref{FPhiC2}), observe that
	\begin{align*}
		\left(\bigotimes_{j=1}^d\varphi_j\right)(\Sigma\xi)&=\prod_{j\notin\mathfrak{Z}}e^{-i\pi\lambda_j^{-1}(\Sigma\xi)_j\cdot(\Sigma\xi)_j}\prod_{j\in\mathfrak{Z}}\delta_0((\Sigma\xi)_j)=\prod_{j\notin\mathfrak{Z}}e^{-i\pi\lambda_j^{-1}(\Sigma\xi)_j\cdot(\Sigma\xi)_j}\delta_{\mathfrak{Z}}(\Sigma\xi),
	\end{align*}
	and
	\begin{align}
		\nonumber	\prod_{j\notin\mathfrak{Z}}e^{-i\pi\lambda_j^{-1}(\Sigma\xi)_j\cdot(\Sigma\xi)_j}&=e^{-i\pi\sum_{j\notin\mathfrak{Z}}(\Sigma\xi)_j\lambda_j^{-1}(\Sigma\xi)_j}=e^{-i\pi(\Sigma\xi)^T\Delta_{C}^{-}(\Sigma\xi)  }\\
		\label{nn1} &=e^{-i\pi(\Sigma^T\Delta_{C}^{-}\Sigma\xi)\cdot\xi}\\
		\nonumber&=e^{-i\pi C^{-}\xi\cdot\xi}\\ 
		\nonumber&=\Phi_{-C^{-}}(\xi).
	\end{align}
This concludes the proof.
\end{proof}

The following corollary generalizes formula \eqref{thM} to the case of singular matrices.

\begin{corollary}\label{corFPhiC}
	Let $M\in\bR^{d\times d}$ and $\cB_{\mathbb{A}_M}(f,g)$ be given as in (\ref{connection}) and define $\eig(MM^T)=\{\lambda_1,\ldots,\lambda_d\}$. Consider the symmetric matrix  $\cM_\cA$ in \eqref{MA}. Let $\Sigma$ be an orthogonal matrix that decomposes $\cM_\cA=\Sigma^T\Delta_{\cM_\cA}\Sigma$ such that $(\Delta_{\cM_\cA})_{jj}=\sqrt{\lambda_j}$ for all $j=1,\ldots,d$ and $(\Delta_{\cM_\cA})_{jj}=-\sqrt{\lambda_j}$ for all $j=d+1,\ldots,2d$. Then,
		\begin{equation}\label{FPhiC2cor}
	\theta_M(\zeta)=|\det(M)|^{-1}\Phi_{-\mathcal{M}_{\cA}^{-}}(\zeta)\delta_{\mathfrak{Z}}(\Sigma\zeta).
	\end{equation}
	If $M$ is invertible, then $\mathcal{M}_{\cA}^{-}=\mathcal{M}_{\cA}^{-1}$, the distribution $\delta_{\mathfrak{Z}}$ does not appear  and we recapture formula \eqref{thM}.
%	\[
%		\theta_M=|\det(M)|^2(\prod_{j\notin\mathfrak{Z'}}e^{i\pi\frac{u^2-v^2)}{\sqrt{\lambda_j}}}\delta_{\mathfrak{Z}})(\Sigma\cdot),
%	\] 
%	where: $\mathfrak{Z}=\mathfrak{Z}'\cup(d+\mathfrak{Z}')$ and $\mathfrak{Z}'=\{j : \lambda_j=0\}$.
	
\end{corollary}
\begin{proof}
	Since $\cM_\cA$ has block decomposition (\ref{MA}), its eigenvalues are $\pm\sqrt{\lambda_j}$, where $\lambda_j$ is the $j$-th eigenvalue of $MM^T$. Let us decompose $\cM_\cA=\Sigma\Delta_{\cM_\cA}\Sigma^T$, with $\Delta_{\cM_\cA}=\diag(\eig({\cM_\cA}))$, in such a way that the eigenvalues of $\cM_\cA$ are ordered as in the assumptions, i.e. $\sqrt{\lambda_1},\ldots,\sqrt{\lambda_d},-\sqrt{\lambda_1},\ldots,-\sqrt{\lambda_d}$. The assertion follows by (\ref{FPhiC2}), expressing the chirp with \eqref{nn1}.
\end{proof}

Recall that if $M\in \bR^{d\times d}$, then $M_{j\ast}$ denotes the $j$-th row of $M$.

\begin{theorem}\label{thmFond}
Let $W_\cA$ be a metaplectic Wigner distribution with $\cA=\pi^{Mp}(\hat\cA)$ having block decomposition (\ref{blockDecA}). Let $A_{13}^{-}$ be the generalized inverse of $A_{13}=\Sigma^T\Delta_{A_{13}}\Sigma$, as in Definition \ref{defGenInv}. Let $\eig(A_{13})=\{\lambda_1,\ldots,\lambda_d\}$ and define $\mathfrak{Z}=\{j:\lambda_j=0\}$. Assume $d\geq2$,  $\mathfrak{Z}\not=\emptyset$ and $\mathfrak{Z}\not=\{\lambda_1,\ldots,\lambda_d\}$, and write $\mathfrak{Z}_1=\{j\in\mathfrak{Z}:(\Sigma A_{11})_{j\ast}=0\}$ and $\mathfrak{Z}_2=\{j\in\mathfrak{Z}:(\Sigma(I_{d\times d}-A_{11}))_{j\ast}=0\}$.  The following statements are equivalent:\\
(i) $W_\cA(f,g)=V_\phi f\overline{V_\psi g}$ for every $f,g\in\cS(\rd)$.\\
(ii) The block decomposition of $\cA$ satisfies
	\begin{equation}\label{finalS}
		\begin{cases}
			A_{13}, A_{21} \qquad \text{symmetric},\\
			A_{21}=A_{11}^TA_{13}^{-}(I_{d\times d}-A_{11}),\\
			(\Sigma(I_{d\times d}-A_{11}))_{j\ast}=0 \quad \vee \quad (\Sigma A_{11})_{j\ast}=0 & \text{if $j\in\mathfrak{Z}$},
		\end{cases}
	\end{equation}
	whereas $\mathfrak{Z}=\mathfrak{Z}_1\cup\mathfrak{Z}_2$ and, in this case, the union is disjoint and %if $\mathfrak{Z}=\mathfrak{Z}'_1\sqcup\mathfrak{Z}_2'$ is any partition of $\mathfrak{Z}$,
	\begin{equation}\label{relWindows}\begin{split}
		&\psi(x)=c_1e^{i\pi A_{13}^{-}(I_{d\times d}-A_{11})x\cdot x}\delta_{\mathfrak{Z}_1}(\Sigma(I_{d\times d}-A_{11})x), \\
		&\phi(y)= c_2e^{-i\pi A_{11}^TA_{13}^{-}y\cdot y}\overline{\delta_{\mathfrak{Z}_2}(-\Sigma A_{11}y)},
		\end{split}
	\end{equation} 
	for suitable constants $c_1,c_2$. {In \eqref{relWindows} if $\mathfrak{Z}_1=\emptyset$ we understand $\delta_{\mathfrak{Z}_1}=1$, if $\mathfrak{Z}_2=\emptyset$ then $\delta_{\mathfrak{Z}_2}=1$.} The cases  $\mathfrak{Z}=\emptyset$ or $\mathfrak{Z}=\{\lambda_1,\ldots,\lambda_d\}$, $d\geq1$, will be treated apart in the subsequent Corollaries  \ref{cor4.8} and \ref{cor4.8bis}. 
\end{theorem}
\begin{proof}
	By (\ref{gspW}) and (\ref{covMWD}), we have that $W_\cA$ is a generalized spectrogram with windows $\phi$ and $\psi$ if and only if
	\[
		W(\cI \psi,\cI\phi)=\cF^{-1}(e^{-i\pi \cM_\cA\zeta\cdot\zeta}).
	\]
This entails
	\begin{align*}
		\psi\otimes\bar \phi &= \cI\widehat{\cA}_{1/2}^{-1}\cF^{-1}(e^{-i\pi \cM_\cA\zeta\cdot\zeta})
		=\cI\mathfrak{T}_W^{-1}\cF_2^{-1}\cF^{-1}(e^{-i\pi \cM_\cA\zeta\cdot\zeta})\\
		&=\cI\mathfrak{T}_W^{-1}\cI_2\cF_1^{-1}(e^{-i\pi \cM_\cA\zeta\cdot\zeta}),
	\end{align*}
	where $\cI_2 F(x,y)=F(x,-y)$. Therefore, using (\ref{TWm1}),
	\begin{align*}
		\psi(x)\overline{\phi(y)}&=\cI\mathfrak{T}_W^{-1}\cI_2\cF_1^{-1}(e^{-i\pi \cM_\cA\zeta\cdot\zeta})(x,y)
		=\mathfrak{T}_W^{-1}\cI_2\cF_1^{-1}(e^{-i\pi \cM_\cA\zeta\cdot\zeta})(-x,-y)\\
		&=\cI_2\cF_1^{-1}(e^{-i\pi \cM_\cA\zeta\cdot\zeta})\left(-\frac{x+y}{2},y-x\right)\\
		&=\cF_1^{-1}(e^{-i\pi \cM_\cA\zeta\cdot\zeta})\left(-\frac{x+y}{2},x-y\right)\\
		&=\int_{\rd}\Phi_{-\cM_\cA}\left(\xi,x-y\right)e^{-2\pi i\frac{x+y}{2}\cdot \xi}d\xi.
	\end{align*}
	Observe that
	\begin{align*}
		\cM_\cA(\xi,x-y)\cdot(\xi,x-y)&=A_{13}\xi\cdot\xi +(I_{d\times d}-2A_{11})(x-y)\cdot\xi-A_{21}(x-y)\cdot(x-y),
	\end{align*}
that gives
	\begin{align*}
		\psi(x)\overline{\phi(y)}&=\Phi_{A_{21}}(x-y)\int_{\rd}\Phi_{-A_{13}}(\xi)e^{-\pi i(I_{d\times d}-2A_{11})(x-y)\cdot\xi}e^{-2\pi i\frac{x+y}{2}\cdot\xi}d\xi\\
		&= \Phi_{A_{21}}(x-y) \int_{\rd}\Phi_{-A_{13}}(\xi)e^{-2\pi i(\frac{x-y}{2}-A_{11}(x-y)+\frac{x+y}{2})\cdot\xi}d\xi \\
		&= \Phi_{A_{21}}(x-y) \int_{\rd}\Phi_{-A_{13}}(\xi)e^{-2\pi i((I_{d\times d}-A_{11})x+A_{11}y)\cdot\xi} d\xi\\
		&=\Phi_{A_{21}}(x-y)\cF\Phi_{-A_{13}}((I_{d\times d}-A_{11})x+A_{11}y).
	\end{align*}
	Let $A_{13}=\Sigma^T\Delta_{A_{13}}\Sigma$, where $\Sigma\in\bR^{d\times d}$ is orthogonal and $\Delta_{A_{13}}=\diag(\lambda_1,\ldots,\lambda_d)$ is the diagonal matrix whose diagonal elements are the eigenvalues $\lambda_1,\ldots,\lambda_d$ of $A_{13}$. Let $A_{13}^{-}=\Sigma^T\Delta_{A_{13}}^{-}\Sigma$ be its generalized inverse. Applying Lemma \ref{lemmaFPhiC} to $C=-A_{13}$,with $s$ being the phase factor,
	\begin{align}
	\nonumber
	\cF\Phi_{-A_{13}}((I_{d\times d}-A_{11})x+A_{11}y)&=s|\det(A_{13})|^{-1/2}\left[\bigotimes_{j=1}^d\varphi_j\right](\Sigma(I_{d\times d}-A_{11})x+A_{11}y)\\
	\label{varphi}
	&=s|\det(A_{13})|^{-1/2}\prod_{j=1}^d\varphi_j((\Sigma((I_{d\times d}-A_{11})x+A_{11}y))_j),
	\end{align}
	where
	\[
		\varphi_j=\begin{cases}
			\Phi_{\lambda_j^{-1}}& \text{if $\lambda_j\neq0$},\\
			\delta_0 & \text{otherwise}
		\end{cases}
	\]
	and $(\Sigma((I_{d\times d}-A_{11})x+A_{11}y))_j$ is the $j$-th coordinate of $\Sigma((I_{d\times d}-A_{11})x+A_{11}y)$. Observe that if $\lambda_j=0$ for every $j=1,\ldots,d$, only the second expression of $\varphi_j$ appears in (\ref{varphi}). 
	
	Let $\mathfrak{Z}:=\{j  : \lambda_j=0\}$. If $j\notin\mathfrak{Z}$, 
	\begin{align*}
		\varphi_j((\Sigma((I_{d\times d}-A_{11})x+A_{11}y))_j)=e^{i\pi\lambda_j^{-1}(\Sigma((I_{d\times d}-A_{11})x+A_{11}y))_j\cdot(\Sigma((I_{d\times d}-A_{11})x+A_{11}y))_j}.
	\end{align*}
	Therefore, %using Remark \ref{remFond},
	\begin{align*}
		\prod_{j\notin\mathfrak{Z}}&\varphi_j((\Sigma((I_{d\times d}-A_{11})x+A_{11}y))_j)\\
		&=\prod_{j\notin\mathfrak{Z}}e^{i\pi\lambda_j^{-1}(\Sigma((I_{d\times d}-A_{11})x+A_{11}y))_j\cdot(\Sigma((I_{d\times d}-A_{11})x+A_{11}y))_j}\\
		&=e^{i\pi(I_{d\times d}-A_{11})^T\Sigma^T\Delta_{A_{13}}^{-}\Sigma(I_{d\times d}-A_{11})x\cdot x}e^{i\pi A_{11}^T\Sigma^T\Delta_{A_{13}}^{-}\Sigma A_{11}y\cdot y}e^{2\pi iA_{11}^T\Sigma^T\Delta_{A_{13}}^{-}\Sigma(I_{d\times d}-A_{11})x\cdot y}.
	\end{align*}
	On the contrary, if $j\in\mathfrak{Z}$,
	\begin{align*}
		\varphi_j((\Sigma((I_{d\times d}-A_{11})x+A_{11}y))_j)&=\delta_0((\Sigma((I_{d\times d}-A_{11})x+A_{11}y))_j)\\
		&=\delta_0((\Sigma(I_{d\times d}-A_{11}))_{j\ast}\cdot x-(\Sigma A_{11})_{j\ast}\cdot y),
	\end{align*}
	where $(\Sigma(I_{d\times d}-A_{11}))_{j\ast}$ denotes the $j$-th row of $\Sigma(I_{d\times d}-A_{11})$, and similarly for $(\Sigma A_{11})_{j\ast}$.
	
	It follows that
	\begin{align*}
		\psi(x)\overline{\phi(y)}&=s|\det(A_{13})|^{-1/2}e^{\pi iA_{21}x\cdot x}e^{-2\pi iA_{21}x\cdot y}e^{i\pi A_{21}y\cdot y}e^{i\pi(I_{d\times d}-A_{11})^T\Sigma^T\Delta_{A_{13}}^{-}\Sigma(I_{d\times d}-A_{11})x\cdot x}\\
		&\qquad\quad\times \quad e^{i\pi A_{11}^T\Sigma^T\Delta_{A_{13}}^{-}\Sigma A_{11}y\cdot y}e^{2\pi iA_{11}^T\Sigma^T\Delta_{A_{13}}^{-}\Sigma(I_{d\times d}-A_{11})x\cdot y}\\
		&\quad\qquad\times\quad \prod_{j\in\mathfrak{Z}}\delta_0((\Sigma(I_{d\times d}-A_{11}))_{j\ast}\cdot x-(\Sigma A_{11})_{j\ast}\cdot y)\\
		&=s|\det(A_{13})|^{-1/2}e^{i\pi (A_{21}+(I_{d\times d}-A_{11})^TA_{13}^{-}(I_{d\times d}-A_{11}))x\cdot x}e^{i\pi(A_{21}+A_{11}^TA_{13}^{-}A_{11})y\cdot y}\\
		&\qquad\quad\times\quad e^{2\pi i(-A_{21}+A_{11}^TA_{13}^{-}(I_{d\times d}-A_{11}))x\cdot y}\\
		&\qquad\quad\times\quad \prod_{j\in\mathfrak{Z}}\delta_0((\Sigma(I_{d\times d}-A_{11}))_{j\ast}\cdot x-(\Sigma A_{11})_{j\ast}\cdot y).
	\end{align*}
	Since the only way to decouple the variables is having $A_{21}-A_{11}^TA_{13}^{-}(I_{d\times d}-A_{11})=0_{d\times d}$ and, for $j\in\mathfrak{Z}$, having either $(\Sigma(I_{d\times d}-A_{11}))_{j\ast}=0$ or $(\Sigma A_{11})_{j\ast}=0$, this shows that $(i)$ is equivalent to the validity of $\mathfrak{Z}=\mathfrak{Z}_1\cup\mathfrak{Z}_2$. 
	This union is disjoint, since if $j\in\mathfrak{Z}_1\cap\mathfrak{Z}_2$, then $(\Sigma A_{11})_{\ast j}=0$ and $(\Sigma(I_{d\times d}-A_{11}))_{\ast j}=\Sigma_{\ast j}=0$, which contradicts the invertibility of $\Sigma$.
	 
	 If $\mathfrak{Z}_1,\mathfrak{Z}_2\neq\emptyset$, then we retrieve (\ref{finalS}) with:
	\begin{align*}
		&\psi(x)=c_1e^{i\pi (A_{21}+(I_{d\times d}-A_{11})^TA_{13}^{-}(I-A_{11}))x\cdot x}{\prod_{j\in\mathfrak{Z}_1}}\delta_0((\Sigma(I_{d\times d}-A_{11}))_{j\ast}\cdot x)\\
		&\phi(y)=c_2e^{-i\pi(A_{21}+A_{11}^TA_{13}^{-}A_{11})y\cdot y}{\prod_{j\in\mathfrak{Z}_2}}\overline{\delta_0}(-(\Sigma A_{11})_{j\ast}\cdot y).
	\end{align*} 
	 Otherwise, if $\mathfrak{Z}_1=\emptyset$, then $\mathfrak{Z}=\mathfrak{Z}_2$ and the product in the definition of $\psi$ does not appear. The case  $\mathfrak{Z}_2=\emptyset$ is similar. Observe that, if the relations in (\ref{finalS}) hold, then
	\begin{align*}
		A_{21}+(I_{d\times d}-A_{11})^TA_{13}^{-}(I_{d\times d}-A_{11})&=A_{21}+A_{13}^{-}(I_{d\times d}-A_{11})-A_{11}^TA_{13}^{-}(I_{d\times d}-A_{11})\\
		&=A_{13}^{-}(I_{d\times d}-A_{11})
	\end{align*}
	and
	\begin{align*}
		A_{21}+A_{11}^TA_{13}^{-}A_{11}=A_{21}-A_{11}^TA_{13}^{-}(I_{d\times d}-A_{11})+A_{11}^TA_{13}^{-}=A_{11}^TA_{13}^{-}.
	\end{align*}
	Therefore, (\ref{relWindows}) follows and the theorem is proved.
\end{proof}

\begin{corollary}\label{cor4.8}
Under the notation of Theorem \ref{thmFond}, assuming that $A_{13}\in GL(d,\bR)$, the following statements are equivalent:\\
(i) $W_\cA(f,g)=V_\phi f\overline{V_\psi g}$ for every $f,g\in\cS(\rd)$. \\
(ii) The block decomposition of $\cA$ satisfies:
\[
	\begin{cases}
		A_{13}, A_{21} \quad \text{symmetric},\\
		A_{21}=A_{11}^T A_{13}^{-1}(I_{d\times d}-A_{11}),\\
	\end{cases}
\]
and for suitable constants $c_1,c_2$, 
\begin{align*}
	&\psi(x)=c_1 \Phi_{A_{13}^{-1}(I_{d\times d}-A_{11})}(x),\\
	&\phi(y)=c_2\Phi_{-A_{11}^TA_{13}^{-1}}(y).
\end{align*}
\end{corollary}
\begin{proof}
	It follows by the proof of Theorem \ref{thmFond}, observing that if $A_{13}\in GL(d,\bR)$, then $\mathfrak{Z}=\emptyset$, $A_{13}^-=A_{13}^{-1}$ and the last condition in (\ref{finalS}) is empty. In fact, \eqref{varphi} reads 
	\[
		\cF\Phi_{-A_{13}}((I_{d\times d}-A_{11})x+A_{11}y)=s|\det(A_{13})|^{-1/2}\Phi_{A_{13}^{-1}}(I_{d\times d}x-A_{11}x+A_{11}y)
	\]
	by applying directly the classical formula \eqref{ft-chirp}. Hence,
	\[
	\psi(x)\overline{\varphi(y)}=s|\det(A_{13})|^{-1/2}\Phi_{A_{21}}(x-y)\Phi_{A_{13}^{-1}}(I_{d\times d}x-A_{11}x+A_{11}y),
	\]
	and, in this case, the decoupling of the variables is equivalent to
	\[
		A_{21}-A_{11}^TA_{13}^{-1}(I_{d\times d}-A_{11})=0_{d\times d}.
	\]
	Finally, 
	\[
		\psi(x)=c_1e^{i\pi A_{13}^{-1}(I_{d\times d}-A_{11})x\cdot x}
	\] 
	and
	\[
		\phi(y)=c_2e^{-i\pi A_{11}^TA_{13}^{-1}y\cdot y}.
	\]
\end{proof}

\begin{corollary}\label{cor4.8bis}
	Under the notation of Theorem \ref{thmFond}, assuming $\mathfrak{Z}=\{1,\dots,d\}$ for $A_{13}$, i.e., $\lambda_j=0$ for every $j=1,\dots,d$. Let  $\mathfrak{Z}_1, \mathfrak{Z}_2$ be defined as in Theorem \ref{thmFond}. Then, the following statements are equivalent:\\
	(i) $W_\cA(f,g)=V_\phi f\overline{V_\psi g}$ for every $f,g\in\cS(\rd)$. \\
	(ii) 
	\[
	(\Sigma(I_{d\times d}-A_{11}))_{j\ast}=0\quad \mbox{or}\quad (\Sigma A_{11})_{j\ast}=0,\quad\mbox{for}\,\, j=1,\dots,d,
	\]
 and
	\begin{align*}
		&\psi(x)= e^{i\pi A_{21}x\cdot x}\delta_{\mathfrak{Z}_1}(\Sigma(I_{d\times d}-A_{11})x),\\
		&\phi(y)=e^{-i\pi A_{21}y\cdot y}\overline{\delta_{\mathfrak{Z}_2}(-\Sigma A_{11}y)},
	\end{align*}
where we understand $\delta_{\mathfrak{Z}_1}=1$ if $\mathfrak{Z}_1=\emptyset$,  $\delta_{\mathfrak{Z}_2}=1$ if $\mathfrak{Z}_2=\emptyset$.
\end{corollary}
\begin{proof}
Again, we refer to the proof of Theorem \ref{thmFond}. Under the assumption $\lambda_j=0$, for every $j=1,\dots,d$, \eqref{varphi} reads as
\[
	\cF\Phi_{-A_{13}}((I_{d\times d}-A_{11})x+A_{11}y)=\prod_{j=1}^d\delta_0((\Sigma((I_{d\times d}-A_{11})x+A_{11}y))_j),
\]
since the part corresponding to the non-zero eigenvalues is missing in the right hand-side of \eqref{varphi}. Thus, we have
	\[
	\psi(x)\overline{\phi(y)}=\Phi_{ A_{21}}(x-y)\prod_{J=1}^{d}\delta_0((\Sigma(I_{d\times d}-A_{11})x+ A_{11}y)_j)\
	\]
	and the equivalence follows by decoupling the variables on the right hand-side. In fact, the decoupling depends only on the distributional part in the right hand-side, and it is possible if and only if $(ii)$ is satisfied.
\end{proof}

\begin{corollary}\label{cor4.9}
	Let $\tau\in\bR$. Then, the (cross-)$\tau$-Wigner distribution $W_\tau$ is a generalized spectrogram if and only if $\tau=0$ or $\tau=1$.
\end{corollary}
\begin{proof}
	Let $\tau\in\bR$. The matrix $A_\tau$ has block decomposition (\ref{Atau}). Therefore, $A_{13}=A_{21}=0_{d\times d}$ for every $\tau\in\bR$.
%	 In particular, $A_{13}^-=0_{d\times d}$, so that the first two equations in \eqref{finalS} are satisfied by $A_\tau$ for every $\tau\in\bR$. 
	We may then apply Corollary \ref{cor4.8bis}. We prove that  condition $(ii)$ in there   is satisfied if and only if $\tau=0,1$.
	 Since $A_{11}=(1-\tau)I_{d\times d}$, the matrices that appear in $(ii)$ read as $\tau\Sigma$ and $(1-\tau)\Sigma$, respectively, and any row of these matrices can be zero if and only if either $\tau=0$ or $\tau=1$. 
	
	We have $\mathfrak{Z}=\{\text{$j$: the $j$-th eigenvalue of $A_{13}$ vanishes}\}=\{1,\ldots,d\}$. If $\tau=0$, then $(\tau\Sigma)_{j\ast}=0$ for every $j\in\{1,\ldots,d\}=\mathfrak{Z}$. If $\tau=1$, $((1-\tau)\Sigma)_{j\ast}=0$ for every $j\in\{1,\ldots,d\}=\mathfrak{Z}$. For other values of $\tau$ the condition $(ii)$ of Corollary \ref{cor4.8bis} is not satisfied. This concludes the proof.
\end{proof}

Since the STFT is a particular example of $\cA$-Wigner distribution, as recalled in Subsection \ref{2.4}, it is natural to extend the definition of generalized spectrogram in \eqref{spect} to the wider class of metaplectic Wigner distributions.

\begin{definition}\label{defMS} Let $W_\cA$ be a metaplectic Wigner distribution with $\cA=\pi^{Mp}(\hat\cA)$. The \textbf{$\cA$-metaplectic spectrogram} with windows $\phi$ and $\psi$  is defined by
	\[
	{\Sp}^{\cA}_{\phi,\psi}(f,g):=W_\cA(f,\phi)\overline{W_\cA(g,\psi)},
	\]
	for $\phi,\psi,f,g$ in suitable functional or distributional spaces.
\end{definition} 

If we choose $\cA=A_{ST}$ in \eqref{AST}, then $W_\cA(f,\phi)=V_\phi f$, $W_\cB(g,\psi)=V_\psi g$ and we go back to the original definition in \eqref{spect}, namely,  ${\Sp}^{A_{ST}}_{\phi,\psi}={\Sp}_{\phi,\psi}$.\\
%If $\cA=A_{\tau_1}$ and $\cB=A_{\tau_2}$, with $\tau_1,\tau_2\in\bR$, we recapture the (rescaled)  parameterized two window spectrogram ${\Sp}^{\tau_1,\tau_2}_{\phi,\psi}$ introduced in \cite[Definition 4]{BdDO4}. 
%Namely, 
%\[
%{\Sp}^{A_{\tau_1},A_{\tau_2}}_{\phi,\psi}(f,g)=W_{A_{\tau_1}}(f,\phi)\overline{W_{A_{\tau_2}}(g,\psi)}=2^d \mathfrak{T}_{(1/2) I_{d\times d}}{\Sp}^{\tau_1,\tau_2}_{\cI\phi,\cI\psi}.
%\]

%In what follows we study the case $\cA=\cB$, devoting the general case to a subsequent investigation.  In this case, we simply write 
%\[
%{\Sp}^\cA_{\phi,\psi}(f,g):=W_\cA(f,\phi)\overline{W_\cA(g,\psi)},
%\]
%and we call the previous expression  \textbf{$\cA$-metaplectic spectrogram} with windows $\phi$ and $\psi$.  
If $W_\cA$ is a shift-invertible metaplectic Wigner distribution, using Theorem \ref{CG-thm1} we easily obtain the following result:
\begin{corollary}\label{5.2}
	Let $W_\cA$ be a shift-invertible metaplectic Wigner distribution as in (\ref{S-inv}). Then,
	\[
		{\Sp}^\cA_{\phi,\psi}(f,g)=|\det(E)|{\Sp}_{\hat S\phi,\hat S\psi}(f,g)\circ E,
	\]
	where $\hat S\in Mp(d,\bR)$, and $E\in GL(2d,\bR))$ are defined in Theorem \ref{CG-thm1}.
\end{corollary}
A natural question is then whether a metaplectic Wigner distribution $W_\cB$ can be written as an $\cA$-metaplectic spectrogram with $W_\cA$ shift-invertible.

 \begin{corollary}\label{5.3}
 	Let $W_\cA$ be a shift-invertible metaplectic Wigner distribution as in (\ref{S-inv}). Then, the equality
 	\begin{equation}\label{e1}
 	{\Sp}^\cA_{\phi,\psi}(f,g)=W_\cB(f,g),
 	\end{equation}
 	holds true for some metaplectic Wigner distribution $W_\cB$ if and only if the projection $\cB=\pi^{Mp}(\hat\cB)$ has a  block decomposition (cf. \eqref{blockDecA}) that  satisfies (\ref{finalS}), and the windows $\hat S\phi$, $\hat S\psi$ satisfy (\ref{relWindows}).
 \end{corollary}
\begin{proof}
	We combine  Theorem \ref{thmFond} and the preceding Corollary \ref{5.2}.
\end{proof}

\section{The one-dimensional case}
In the case dimension $d=1$ the expression of the preceding results simplifies and we may prove that the class of shift-invertible distributions in Theorem \ref{CG-thm1}  has empty intersection with the class of generalized spectrograms $W_\cA$ considered in Section $4$. Furthermore, the union of these two classes exhausts the $W_\cA$ of Cohen's class. We recall from the Introduction that a metaplectic Wigner distribution $W_\cA$ belongs to the Cohen's class if and only if
\begin{equation}\label{6.1}
W_\cA(f,g)=W(f,g)\ast \cF^{-1}\lambda
\end{equation}
where, with the notation of the block decomposition \ref{blockDecA}, for $d=1$, 
\begin{equation}\label{6.2}
	\lambda\phas=e^{-\pi i(A_{13}x^2+(2A_{11}-1)x\xi -A_{21}\xi^2)}.
\end{equation}
\begin{theorem}\label{teor6.1}
Let $W_\cA$ be a metaplectic Wigner distribution, written as in \eqref{6.1} and \eqref{6.2} with $A_{11}, A_{13}, A_{21}\in\bR$. Then, the following are equivalent:
\begin{itemize}
	\item[(i)] $A_{11}(1-A_{11})-A_{21}A_{13}=0$.
		\item[(ii)] $W_\cA$ is a generalized spectrogram.
			\item[(iii)] $W_\cA$ is not shift-invertible.
	\end{itemize}
\end{theorem}
\begin{proof}
	The equivalence (i)$ \Leftrightarrow $(iii) follows from Remark $2.20$ in \cite{CR2022}, where it was observed that $W_\cA$ of Cohen's class is shift-invertible if and only if the matrix 
	\begin{equation}\label{E-covariant}
		E_\cA=\left(\begin{array}{cc}
			A_{11} & 	A_{13}\\
			A_{21} & I_{d\times d}-	A_{11}
		\end{array}\right)
	\end{equation}
is invertible. The condition (i) reads $\det E_\cA=0$, hence it is equivalent to non-shift-ivertibility. To prove the equivalence (i)$\Leftrightarrow $(ii) we distinguish two cases:\\
$\bullet$ $A_{13}\not=0$.\\
  In this case (i) is equivalent to $A_{21}=A_{13}^{-1} A_{11}(1-A_{11})$. Since $A_{13}\not=0$ we may argue under the assumptions of Corollary \ref{cor4.8}, where its condition $(ii)$ is satisfied for dimension $d=1$,  and we have the equivalence with the spectrogram property,   for suitable windows.\\
  $\bullet$ $A_{13}=0$.\\
   In this case (i)  holds if and only if $A_{11}=0$ or $1-A_{11}=0$. Since  $A_{13}=0$ the assumptions of  Corollary \ref{cor4.8bis} are satisfied and in the 1-dimensional case the condition $(ii)$ in Corollary \ref{cor4.8bis} reduces exactly to the alternative $A_{11}=0$ or $1-A_{11}=0$. So the equivalence with condition (ii) is proved also in this case.
\end{proof}
The validity of Theorem \ref{6.1} extends to the more general spectrograms in Corollary \ref{5.3} in dimension $d=1$. We omit a higher dimensional analysis, because of the complexity of the matrix presentation.
\begin{remark}\label{rem6.2} Concerning boundedness of spectrograms in function spaces, let us just observe that in \eqref{spect}, Definition \ref{def1.2}, each factor of the product $V_\phi f\overline{V_\psi g}$ possesses the good continuity property of the STFT, see for example \cite{Elena-book,book}. The same applies to spectrograms in Corollary \ref{5.3},
	\begin{equation}\label{eq6.3}
	W_\cB(f,g)=	{\Sp}^{\cA}_{\phi,\psi}(f,g)=W_\cA(f,\phi)\overline{W_\cA(g,\psi)}
	\end{equation}
where, for $W_\cA$ shift-invertible, we have for each factor on the right hand-side results in modulation spaces, cf. \cite{CR2022}, and in Lebesgue spaces, see Section \ref{sec:BSI} above. Boundedness properties for $W_\cB$ in \eqref{eq6.3} can then be deduced from the corresponding algebra properties of the involved spaces. The continuity results for $W_\cB$ are weaker than those for $W_\cA$, as also expected from $\mathrm{(ii)}$ in Theorem \ref{teor6.1}, since shift-invertibility is lost for $W_\cB$. We omit detailed results.
\end{remark}
\begin{remark}\label{rem6.3}
	The windows $\phi,\psi$ required to represent the spectrogram as metaplectic Wigner distributions belong to $\cS'(\rd)$. In the applications more regular windows are used, in particular $\phi$ and $\psi$ are of Gaussian type, with exponential decay. The corresponding spectrograms can be written in the form of metaplectic Wigner distributions as well, however this requires the use of complex-valued symplectic matrices $\cA$, as we shall detail in a future paper.
\end{remark}
\section*{Acknowledgements}
The authors have been supported by the Gruppo Nazionale per l’Analisi Matematica, la Probabilità e le loro Applicazioni (GNAMPA) of the Istituto Nazionale di Alta Matematica (INdAM). The second author was supported by the University of Bologna and HES-SO Valais - Wallis School of Engineering. He also acknowledge the support of The Sense Innovation and Research Center, a joint venture of the University of Lausanne (UNIL), The Lausanne University Hospital (CHUV), and The University of Applied Sciences of Western Switzerland – Valais/Wallis (HES-SO Valais/Wallis).


\begin{thebibliography}{10}

\bibitem{AGR2016}
L.~D. Abreu, K.~Gr{\"o}chenig, and J.~L. Romero.
\newblock On accumulated spectrograms.
\newblock {\em Trans. Am. Math. Soc.}, 368(5):3629--3649, 2016.

\bibitem{KB2020}
K.~A.~O. \'Arp\'ad~B\'enyi.
\newblock {\em Modulation Spaces With Applications to Pseudodifferential
  Operators and Nonlinear {S}chr\"odinger Equations}.
\newblock Applied and Numerical Harmonic Analysis. Birkh\"auser New York, NY,
  2020.

\bibitem{bastianoni2023tauquantization}
F.~Bastianoni and F.~Luef.
\newblock $\tau$-quantization and $\tau$-{C}ohen classes distributions of
  {F}eichtinger operators.
\newblock {\em J. Pseudo-Differ. Oper. Appl.}, 15(4):86, 2024.

\bibitem{BCGT2020}
D.~{B}ayer, E.~{C}ordero, K.~{G}r{\"o}chenig, and S.~I. {T}rapasso.
\newblock {L}inear perturbations of the {W}igner transform and the
{W}eyl
quantization.
\newblock In {\em {A}dvances in {M}icrolocal and {T}ime-{F}requency
{A}nalysis}, pages 79--120. {S}pringer, 2020.

\bibitem{bcoquadratic}
P.~{B}oggiatto, E.~{C}arypis, and A.~{O}liaro.
\newblock {W}igner representations associated with linear
transformations of
the time-frequency plane.
\newblock In {\em {P}seudo-differential operators: analysis,
applications and
computations. {S}elected papers based on lectures presented at the
meeting of
the {I}{S}{A}{A}{C} {G}roup in {P}seudo-{D}ifferential {O}perators
({I}{G}{P}{D}{O}), {L}ondon, {U}{K}, {J}uly 13--18, 2009}, pages
275--288.
2011.

\bibitem{BdDO4}
P.~Boggiatto, B.~K. Cuong, G.~De~Donno, A.~Oliaro,
\newblock Generalized spectrograms and $\tau$-{W}igner transforms.
\newblock {\em Cubo}, 12(3):171--185, 2010.

\bibitem{BdDO3}
P.~Boggiatto, G.~De~Donno, and A.~Oliaro.
\newblock {\em A Class of Quadratic Time-frequency Representations Based on the
  Short-time Fourier Transform}, pages 235--249.
\newblock Birkh\"auser Basel, Basel, 2007.

\bibitem{BdDO1}
P.~Boggiatto, G.~De~Donno, and A.~Oliaro.
\newblock Weyl quantization of {L}ebesgue spaces.
\newblock {\em Math. Nachr.}, 282(12):1656--1663, 2009.

\bibitem{BdDO2}
P.~Boggiatto, G.~De~Donno, and A.~Oliaro.
\newblock Time-frequency representations of {W}igner type and
  pseudo-differential operators.
\newblock {\em Trans. Amer. Math. Soc.}, 362(9):4955--4981, 2010.

\bibitem{boghuds}
P.~Boggiatto, G.~De~Donno, and A.~Oliaro.
\newblock Hudson's theorem for $\tau$-{W}igner transforms.
\newblock {\em Bull. Lond. Math. Soc.}, 45(6):1131--1147, 07 2013.

\bibitem{BDOJMAA2007}
P.~Boggiatto, G.~D. Donno, and A.~Oliaro.
\newblock Uncertainty principle, positivity and ${L}^p$-boundedness for
  generalized spectrograms.
\newblock {\em J. Math. Pures Appl.}, 335(1):93--112, 2007.

\bibitem{Cohen1}
L.~Cohen.
\newblock Generalized phase-space distribution functions.
\newblock {\em J. Math. Phys.}, 7(5):781--786, 05 1966.

\bibitem{Cohen2}
L.~Cohen.
\newblock {\em Time-frequency analysis: Theory and Applications}, volume 778.
\newblock Prentice Hall PTR New Jersey, 1995.

\bibitem{cdet18}
E.~Cordero, L.~D'Elia, and S.~I. Trapasso.
\newblock Norm estimates for $\tau$-pseudodifferential operators in {W}iener
  amalgam and modulation spaces.
\newblock {\em J. Math. Anal. Appl.}, 471(1):541--563, 2019.

\bibitem{CGshiftinvertible}
E.~Cordero and G.~Giacchi.
\newblock Symplectic analysis of time-frequency spaces.
\newblock {\em J. Math. Pures Appl.}, 177:154--177, 2023.

\bibitem{CG-Excursus}
E.~Cordero and G.~Giacchi.
\newblock Excursus on modulation spaces via metaplectic operators and related
  time-frequency representations.
\newblock {\em Sampl. Theory Signal Process. Data Anal.}, 22(1):9, 2024.

\bibitem{cnt18}
E.~Cordero, F.~Nicola, and S.~I. Trapasso.
\newblock Almost diagonalization of $\tau$-pseudodifferential operators with
  symbols in {W}iener amalgam and modulation spaces.
\newblock {\em J. Fourier Anal. Appl.}, 25(4):1927--1957,
  2019.

\bibitem{Elena-book}
E.~Cordero and L.~Rodino.
\newblock {\em Time-Frequency Analysis of Operators}.
\newblock De Gruyter, Berlin, Boston, 2020.

\bibitem{CR2021}
E.~Cordero and L.~Rodino.
\newblock Wigner analysis of operators. {P}art {I}: Pseudodifferential operators
  and wave fronts.
\newblock {\em Appl. Comput. Harmon. Anal.}, 58:85--123, 2022.

\bibitem{CR2022}
E.~Cordero and L.~Rodino.
\newblock Characterization of modulation spaces by symplectic representations
  and applications to {S}chr\"odinger equations.
\newblock {\em J. Funct. Anal.}, 284(9):109892, 2023.

\bibitem{CT2020}
E.~Cordero and S.~I. Trapasso.
\newblock Linear perturbations of the {W}igner distribution and the {C}ohen
  class.
\newblock {\em Anal. Appl. (Singap.)}, 18(03):385--422, 2020.

\bibitem{deGossonWigner}
M.~A. de~Gosson.
\newblock {\em The Wigner Transform}.
\newblock World Scientific (Europe), 2017.

\bibitem{Gos11}
M.~A. de~Gosson.
\newblock {\em Symplectic Methods in Harmonic Analysis and in Mathematical
  Physics}.
\newblock Pseudo-Differential Operators. Birk\"auser Basel, 1 edition, July
  2011.

\bibitem{feichtinger-modulation}
H.~G. Feichtinger.
\newblock Modulation spaces on locally compact abelian groups.
\newblock In M.~Krishna, R.~Radha, and S.~Thangavelu, editors, {\em Wavelets
  and Their Applications}, pages 99--140. Allied Publishers, 2003.

\bibitem{folland89}
G.~B. Folland.
\newblock {\em Harmonic Analysis in Phase Space}.
\newblock Annals of Mathematics Studies. Princeton University Press, 1989.

\bibitem{FollandSit89}
G.~B. Folland and A.~Sitaram.
\newblock The uncertainty principle: a mathematical survey.
\newblock {\em J. Fourier Anal. Appl.}, 3(3):207--238, 1997.

\bibitem{Fuhr}
H.~F{\"u}hr and I.~Shafkulovska.
\newblock The metaplectic action on modulation spaces.
\newblock {\em Appl. Comput. Harmon. Anal.}, 68:101604, 2024.

\bibitem{GalleaniCohen}
L.~Galleani and L.~Cohen.
\newblock The {W}igner distribution for classical systems.
\newblock {\em Physics letters. A}, 302(4):149--155, 2002.

%\bibitem{Giacchi}
%G.~Giacchi.
%\newblock Metaplectic {W}igner distributions.
%\newblock Springer INdAM Series (To Appear), 2022.

  \bibitem{Giacchi}
G.~Giacchi.
\newblock Metaplectic {W}igner distributions.
\newblock In {\em New Trends in Complex and Fourier Analysis, and Operator Theory}, Tabacco, A., Peloso M., Arcozzi N. Springer Nature Singapore Pte Ltd., to appear.

\bibitem{book}
K.~{G}r{\"o}chenig.
\newblock {\em {F}oundations of {T}ime-{F}requency {A}nalysis}.
\newblock {A}ppl. {N}umer. {H}armon. {A}nal. {B}irkh{\"a}user, {B}oston,
{M}{A}, 2001.

\bibitem{hlawbook}
%F.~Hlawatsch and F.~Auger.
%\newblock {\em Time-Frequency Analysis}.
%\newblock John Wiley \& Sons, 2013.
{H}lawatsch, {F}ranz, and {F}ran{\c c}ois Auger, eds. {T}ime-{f}requency {a}nalysis. John Wiley \& Sons, 2013.

\bibitem{hudson74}
R.~Hudson.
\newblock When is the {Wi}gner quasi-probability density non-negative?
\newblock {\em Rep. Mathematical Phys.}, 6(2):249--252, 1974.

\bibitem{janssenhuds84}
A.~J. E.~M. Janssen.
\newblock A note on {H}udson's theorem about functions with nonnegative wigner
  distributions.
\newblock {\em SIAM J. Math. Anal.}, 15(1):170--176, 1984.

\bibitem{Janssen91}
A.~J. E.~M. Janssen.
\newblock Optimality property of the {G}aussian window spectrogram.
\newblock {\em IEEE Trans. Signal Process.}, 39(1):202--204, 1991.

\bibitem{janssenposspread97}
A.~J. E.~M. Janssen.
\newblock Positivity and spread of bilinear time-frequency distributions.
\newblock In {\em Wigner distribution}, pages 1--58. Elsevier Sci. B. V.,
  Amsterdam, 1 edition, 1997.
  
  \bibitem{toftbil17}
J.~Toft.
\newblock Matrix parameterized pseudo-differential calculi on modulation
  spaces.
\newblock In {\em Oberguggenberger}, pages 215--235. M., Toft, J., Vindas, J., Wahlberg, P. (eds) Generalized Functions and Fourier Analysis. Operator Theory: Advances and Applications(), vol 260. Birkh\"auser, Cham., 2017.

%\bibitem{toftbil17}
%J.~Toft.
%\newblock Matrix parameterized pseudo-differential calculi on modulation
%  spaces, pages 215--235.
%\newblock Birkhäuser, Cham, 

\bibitem{Ville48}
J.~Ville.
\newblock Th\'eorie et applications de la notion de signal analytique.
\newblock {\em CeT, Laboratoire de Telecommunications de la Societe Alsacienne
  de Construction Mecanique}, 2:61--74, 1948.

\bibitem{Wigner32}
E.~Wigner.
\newblock On the quantum correction for thermodynamic equilibrium.
\newblock {\em Phys. Rev.}, 40:749--759, Jun 1932.

\bibitem{WongWeylTransform1998}
M.~Wong.
\newblock {\em The {W}eyl Transform}, pages 19--24.
\newblock Springer New York, New York, NY, 1998.

\bibitem{Zhang21bis}
Z.~Zhang.
\newblock Uncertainty principle of complex-valued functions in specific free
  metaplectic transformation domains.
\newblock {\em J. Fourier Anal. Appl.}, 27(4):68, 2021.

\bibitem{Zhang21}
Z.~Zhang, X.~Shi, A.~Wu, and D.~Li.
\newblock Sharper ${N}$-{D} {H}eisenberg's uncertainty principle.
\newblock {\em IEEE Signal Process. Lett.}, 28:1665--1669, 2021.

\bibitem{ZJQ21}
Z.-C. Zhang, X.~Jiang, S.-Z. Qiang, A.~Sun, Z.-Y. Liang, X.-Y. Shi, and A.-Y.
  Wu.
\newblock Scaled {W}igner distribution using fractional instantaneous
  autocorrelation.
\newblock {\em Optik}, 237:166691, 2021.

\end{thebibliography}
\end{document}